\documentclass[12pt,fleqn]{article}

\usepackage{mathtools,amsfonts,amsmath, amsthm,amssymb}
\usepackage{csquotes}

\usepackage{pdfsync}

 \usepackage[utf8]{inputenc}
\usepackage[T1]{fontenc}
 \usepackage{fouriernc}
\usepackage{mathrsfs} 
\usepackage{stmaryrd}
\usepackage{bbding}
 \usepackage{latexsym}
\usepackage{esint}

\usepackage[margin=35mm]{geometry} 
 \swapnumbers
 \numberwithin{equation}{section}

\def\bp{\begin{proof}}
\def\ep{\end{proof}}
\usepackage{color}
 \definecolor{db}{rgb}{0.0,0.0,0.8} 
\definecolor{dg}{rgb}{0.0,0.55,0.14}
\definecolor{dr}{rgb}{0.5,0,0.07}

  \usepackage{hyperref}
 \hypersetup{frenchlinks=true,colorlinks=true,linkcolor=db,citecolor=dg,
 bookmarks=true}

 \usepackage{enumerate}

 \def\D{{\mathbb D}}
 
\newcommand{\hwthm}{\vskip 0.3cm
     \noindent\textbf{Loose Theorem.${}\ {}$}
    }

\newtheorem{theorem}{Theorem}[section]
\newtheorem{proposition}[theorem]{Proposition}
\newtheorem{lemma}[theorem]{Lemma}
\newtheorem{corollary}[theorem]{Corollary}
\theoremstyle{definition}

\theoremstyle{definition}

\theoremstyle{definition}

\theoremstyle{definition}

\theoremstyle{definition}

\theoremstyle{definition}
\newtheorem{remark}[theorem]{Remark}
\theoremstyle{definition}

\newcounter{step}

\def\be{\begin{equation}}
\def\ee{\end{equation}}
\def\bes{\begin{equation*}}
\def\ees{\end{equation*}}
\def\bt{\begin{theorem}}
\def\et{\end{theorem}}
\def\bpr{\begin{proposition}}
\def\epr{\end{proposition}}
\def\bl{\begin{lemma}}
\def\el{\end{lemma}}
\def\bc{\begin{corollary}}
\def\ec{\end{corollary}}
\def\br{\begin{remark}}
\def\er{\end{remark}}
\def\ben{\begin{enumerate}}
\def\bena{\begin{enumerate}[a)]}
\def\een{\end{enumerate}}
\def\bit{\begin{itemize}}
\def\iit{\end{itemize}}

\def\det{\operatorname{det}}

\def\deg{\operatorname{deg}}
\def\div{\operatorname{div}}

\def\Re{\operatorname{Re}}
\def\Im{\operatorname{Im}}

\def\ind{\operatorname{ind}}
\def\tr{\operatorname{tr}\ }

\DeclareMathAlphabet{\mathonebb}{U}{bbold}{m}{n}

\def\R{{\mathbb R}}
\def\N{{\mathbb N}}

\def\C{{\mathbb C}}
\def\Z{{\mathbb Z}}

\def\S{{\mathbb S}}

\def\fo{\forall\, }

\def\va{\varphi}
\def\d{\displaystyle}
\def\im{\imath}
\def\ve{\varepsilon}
\def\p{\partial}
\def\l{\label}
\def\O{\Omega}
\def\po{\partial\Omega}

\def\na{\nabla}

\def\so{{\mathbb S}^1}

\date{October 28, 2013}
\title{Existence of critical points with semi-stiff boundary conditions for singular perturbation problems in  simply connected planar domains}
\author{Xavier Lamy \thanks{Universit\'e de Lyon,  CNRS UMR 5208, Universit\'e Lyon 1, Institut Camille Jordan, 43 blvd. du 11 novembre 1918, F-69622 Villeurbanne cedex, France. Email address: xlamy$@$math.univ-lyon1.fr}       \and
Petru Mironescu    \thanks{Universit\'e de Lyon,  CNRS UMR 5208, Universit\'e Lyon 1, Institut Camille Jordan, 43 blvd. du 11 novembre 1918, F-69622 Villeurbanne cedex, France. Email address: mironescu$@$math.univ-lyon1.fr}  
}

\begin{document}

\maketitle
 \begin{abstract}
 Let $\O$ be a smooth bounded simply connected domain in $\R^2$. We investigate the existence of critical points of the energy $E_\ve(u)=1/2\int_\O |\nabla u|^2+1/(4\ve^2)\int_\O (1-|u|^2)^2$, where  the complex map $u$ has modulus one and prescribed degree $d$  on the boundary. Under suitable nondegeneracy assumptions on $\O$, we prove existence of critical points for small $\ve$. More can be said when the prescribed degree equals one. First, we obtain existence of critical points  in domains close to a disc. Next, we prove that critical points exist in \enquote{most} of the domains.
 \end{abstract}
%
%
%

\section{Introduction}
\label{sin}

Let $\O\subset\R^2$ be a smooth bounded simply connected domain. Let a map $u$ belong to the space
\[\mathcal{E}:=\left\lbrace u\in H^1(\O,\C);\, |\tr u|=1\right\rbrace,\]
where $\tr u$ denotes the trace of $u$ on the boundary $\po$. Then the trace $\tr u$ of $u$ on $\po$ belongs to the space $H^{1/2}(\po;\so)$, and therefore we can define its winding number or degree, which we denote by $\deg(u,\po)$ (see \cite[Appendix]{boutetgeorgescupurice}; see also \cite[Section 2]{bmrs} for more details). This allows us to define the class
\[\mathcal{E}_d=\left\lbrace u\in H^1(\O;\C);\,|\tr u|=1,\,\deg(u,\po)=d\right\rbrace.\]
In this paper we study the existence of critical points of the Ginzburg-Landau energy functional
\begin{equation*}
E_\ve(u)=\frac{1}{2}\int_\O |\nabla u|^2+\frac{1}{4\ve^2}\int_\O (1-|u|^2)^2
\end{equation*} 
in the space $\mathcal{E}_d$,  i.e., of critical points with prescribed degree $d$. More specifically, we are interested in non trivial critical points, that is critical points which are not constants of modulus one.

The  prescribed degree boundary  condition is an  intermediate model between the  Dirichlet and the Neumann boundary conditions. The asymptotic of minimizers of the Ginzburg-Landau energy $E_\ve$ with Dirichlet boundary condition was first studied by Bethuel, Brezis and H\'elein in their classical work \cite{bbh}. In particular, it was shown in \cite{bbh} that minimizers $u_\ve$ have zeros \enquote{well-inside} $\O$, and that these zeros approach the singularities (vortices) of the limit $u_\ast$ of the $u_\ve$'s as $\ve\to 0$. In contrast, the only minimizers of $E_\ve$ with no boundary condition are constants. The same holds even for stable critical points of $E_\ve$ with Neumann boundary conditions \cite{serf_stab}. The analysis of the prescribed degrees boundary condition (in domains which may be multiply  connected) leads to a richer global picture \cite{ber_voss},  \cite{gbcvpde}, \cite{bmjfa}, \cite{bgrnonexistence},  \cite{brjems}, \cite{d_s}, \cite{bmrs}. More specifically, in multiply connected domains minimizers of $E_\ve$  may exist  \cite{gbcvpde}, \cite{bmjfa} or not \cite{bgrnonexistence}. However, in such domains critical points of $E_\ve$  always exist \cite{brjems}, \cite{d_s}. In  simply connected domains, minimizers never exist \cite{bmjfa}. More involved is the study of the existence of critical points in simply connected domains; this is our purpose. Typical methods in absence of absolute minimizers consist in constructing local minimizers, or in constructing critical points by minimax methods. Construction of local minimizers proved to be successful in multiply connected domains \cite{brjems}, but the arguments there do not adapt to our case. Minimax techniques led in \cite{bmrs} to the proof of the existence of critical points in simply connected domains for {\it large} $\ve$, but again these techniques do not seem to work for {\it small} $\ve$. 

The present paper is devoted to the existence of critical points for {\it small} $\ve$ and thus complements \cite{bmrs}. Our approach relies on singular perturbations techniques, in the spirit of Pacard and Rivi\`ere \cite{pacardriviere}.
We explain this approach in the special case   where the prescribed degree is $d=1$.  We first recall the main result in \cite{bbh}. Consider the minimization of $E_\ve$ with Dirichlet boundary condition:
\bes
\min\{ E_\ve(u);\, \tr u=g\text{ on }\p\O\}.
\ees
Here, $g:\p\O\to\so$ is smooth, and we  assume that   $\deg (g, \p\O)=1$. Then there exists some $a\in\O$ such that, possibly up to a subsequence, minimizers $u_\ve$ satisfy $u_\ve\to u_*$, with 
\be
\l{eaa1}
u_*(z)=u_{\ast, a, g}(z)=\d\frac {z-a}{|z-a|}e^{\im H},\ \text{with }H=H_{a,g}\text{ harmonic}. 
\ee
In \eqref{eaa1}, the function $H$ is  uniquely determined (mod $2\pi$) by the condition
\be
\l{eaa2}
u_*=g\text{ on $\p\O$}.
\ee 
The point $a$ is not arbitrary: it has to be a critical point (actually, a point of minimum) of the \enquote{renormalized energy} $W(\cdot, g)$ associated with $g$.

 In order to explain our main results in the case of prescribed degree boundary condition, we perform a handwaving  analysis of our problem when $d=1$. Assume that $u_\ve$ is a critical point of $E_\ve$ in $\mathcal{E}_1$. Then $u_\ve$ has to vanish at some point $a_\ve$, and up to a subsequence we have either 
\ben[(i)]
\item
$a_\ve\to a\in\O$

or
\item
$a_\ve\to a\in\p\O$.
\een
Assume that (i) holds. Assume further, for the purpose of our discussion, that $a_\ve$ is the only zero of $u_\ve$. Then the analysis in \cite{bbh} suggests that the limit $u_*$ of the $u_\ve$'s should be again  of the form $u_*(z)=\d\frac {z-a}{|z-a|}e^{\im\psi}$. Formally, the fact that $u_\ve$ is a critical point of $E_\ve$ leads, as in \cite{bbh}, to the conclusion that   the limiting point $a$ is a critical point of a suitable renormalized energy $\widehat W(\cdot)$. Some basic properties of the energy $\widehat W$ are studied in \cite{lefterradulescu}; we will come back to this in Section \ref{sre}. Of interest to us is the fact that $\widehat W$ is smooth and does have critical points.  

Let $a$ be a critical point of $\widehat W$, and let $u_\ast$ be as in \eqref{eaa1}-\eqref{eaa2}. We plan to construct critical points $u_\ve$ of $E_\ve$ in ${\cal E}_1$ such that $u_\ve\to u_\ast$ as $\ve\to 0$. 
The construction goes as follows: we first construct, for a \enquote{generic} boundary datum $g:\po\to\so$, critical points of $E_\ve$ with Dirichlet boundary condition $g$. We next prove that, for some appropriate $g=g_\ve$, the corresponding critical point $u_\ve$ is actually a critical point of $E_\ve$ with prescribed degree $1$. Construction of critical points of $E_\ve$ with Dirichlet boundary conditions was already considered in the literature. Such construction can be performed by either variational methods \cite{linlin} (see also \cite{delpinofelmer}), or  inverse functions methods \cite{pacardriviere} (see also \cite{delpinomussokowalczyk}). Our approach is inspired by the one of Pacard and Rivi\`ere \cite{pacardriviere}. In \cite{pacardriviere}, critical points of $E_\ve$ with Dirichlet boundary condition $g$ are constructed under a  nondegeneracy assumption for the corresponding renormalized energy $W(\cdot, g)$. We encounter a similar situation in our problem: we are able to construct critical points of $E_\ve$ under some nondegeneracy assumptions that we  explain below.

To start with,   we will see in Section \ref{sre} that we may associate with each point $a\in\O$ a natural boundary datum $g^a$, solution of the minimization problem
\bes
\min\{ W(a, g);\,  g : \po\to\so,\ \deg (g, \po)=1\}.
\ees
It turns out that, if $a$ is a critical point of $\widehat W$, then $a$ is also a critical point of $W(\cdot, g^a)$ (Section \ref{sre}). Since $\widehat W$ has a global maximum (Section \ref{gen}), $\widehat W$ has critical points, and thus there exists some $a\in\O$ critical point of $W(\cdot, g^a)$. Our first nondegeneracy assumption is 
\bes
(ND1)\ \text{there exists some }a\in\O\text{  nondegenerate critical point of }W(\cdot, g^a).
\ees
Assuming that {\it (ND1)} holds, set $g_0:=g^a$. Then we may prove that, for each $g$ \enquote{close} to $g_0$ in a suitable sense, $W(\cdot, g)$ has a critical point $a(g)$ close to $a$ (Section \ref{snd}). Thus, to such $g\in C^{1,\beta}(\po ; \so)$ we may associate the function 
\bes
T_\ast(g)\in \dot C^\beta(\po ; \R),\ 
T_\ast(g):=\d u_\ast\wedge\frac{\p u_\ast}{\p\nu}, 
\ees
where $u_\ast=u_{\ast, a(g), g}$ is given by \eqref{eaa1}-\eqref{eaa2}. One may prove that the map $g\mapsto T_\ast(g)$ is $C^1$  near $g_0$, and that its differential $L$ at  $g_0$ is a Fredholm operator of index one (Section \ref{sra}). Our second nondegeneracy assumption is
\bes
(ND2)\ L \text{ is onto}.
\ees

We may now  state our first result.  
\bt
\l{tet}
Assume that {\it (ND1)} and {\it (ND2)}  hold.  Then, for small $\ve$, $E_\ve$ has critical points $u_\ve$ with prescribed degree one.
\et
A similar result holds for an arbitrary prescribed degree $d$.

The conditions {\it (ND1)} and {\it (ND2)} seem to be \enquote{generic}.\footnote{Critical points of smooth functionals are \enquote{generically} nondegenerate, and Fredholm operators of index one are \enquote{generically} onto.}
  However, it is not clear whether the assumptions {\it (ND1)} and {\it (ND2)} are ever satisfied. Therefore, our next task is  to exhibit   nondegeneracy situations.
\hwthm
{\it Assume that $d=1$ and that $\O$ is \enquote{close} to a disc. Then {\it (ND1)} and {\it (ND2)} hold. In particular, for small $\ve$, $E_\ve$ has  critical points of prescribed degree $1$.
}

\medskip
\noindent
The above theorem applies to the unit disc $\D$. However, no sophisticated argument is needed for a disc.  Indeed, when $\O=\D$  it is possible to construct explicit  hedgehog type critical points of $E_\ve$ by minimizing $E_\ve$ in the class  of the maps of the form $f(|z|)\d\frac z{|z|}$.

Concerning the existence of critical points of $E_\ve$ in arbitrary domains, we do not know whether {\it (ND1)} and {\it (ND2)} do always hold. However, we have the following result.
\hwthm
{\it
Assume that $d=1$. Then every $\O$ can be approximated with domains satisfying (ND1)-(ND2).
}

\medskip
Our paper contains the proof of the three above theorems, as well as generalizations to higher degrees $d$ and a discussion about the \enquote{generic} nature of our results. The plan of the paper is the following.   In Section \ref{sre}, we recall the definition and the main properties of the renormalized energies corresponding to either Dirichlet or prescribed degree boundary condition, and establish few new properties. In Sections \ref{st} and \ref{sef}, we derive new useful formulas for the renormalized energies. In Section \ref{snd}, we prove that nondegeneracy of critical points of $W(\cdot, g)$ is stable with respect to small perturbations of $g$. Section \ref{prus} is devoted to the proof of a variant of the Pacard-Rivi\`ere \cite{pacardriviere} construction of critical points with Dirichlet condition; this is a key step in our proof. We prove Theorem \ref{tet} (for arbitrary degrees $d$) in Section \ref{scp}. The proof relies on a Leray-Schauder degree argument, and the  corresponding key estimate is obtained in Section \ref{sconv}.  In Section \ref{sno}, we prove that the couple of conditions {\it (ND1)}-{\it (ND2)} is stable with respect to small perturbations of the domain. This and the fact that $\O=\D$ satisfies  {\it (ND1)}-{\it (ND2)} (Section \ref{sra}) implies (a rigorous form of)  the first Loose Theorem. We finally discuss in Section \ref{gen} the \enquote{generic} nature of our results, and establish (a rigorous form of)  the second Loose Theorem.%

\subsection*{Notation}
\ben
\item
Points in $\R^2$ are denoted $z$ in the Sections \ref{st} and \ref{sef} relying on complex analysis techniques, and $x$ or $y$ elsewhere.
\item
$\D(z,r)$, $\overline\D (z,r)$ and $C(z,r)$ denote respectively  the open disc, the closed disc and the circle of center $z$ and radius $r$. We let $\D=\D(0,1)$ denote the unit disc and set $\D_r=\D(0,r)$. $\so$ is the unit circle.
\item
$\wedge$ stands for the vector product of complex numbers or vectors. Examples:
$
(u_1+\im u_2)\wedge (v_1+\im v_2)=u_1v_2-u_2v_1$, $ (u_1+\im u_2)\wedge (\na v_1+\im \na v_2)=u_1\na v_2-u_2\na v_1$, $(\na u_1+\im \na u_2)\wedge (\na v_1+\im \na v_2)=\na u_1\cdot\na v_2-\na u_2\cdot\na v_1$.
\item
If $A$ is a set and $k$ an integer, then we let 
\bes
A_*^k=\{ a=(a_1,\ldots, a_k)\in A^k;\, a_j\neq a_l, \fo j\neq l\}.
\ees
\item
When $k=1$, we identify a collection $a=(a_1)$ with (the  point or number) $a_1$. 
\item
Additional indices emphasize the dependence of objects on variables. E.g.: $\psi_a=\psi_{a,g}$ recalls that $\psi$ depends not only on $a$, but also on $g$.
\een

\tableofcontents

\section{Renormalized energies and canonical maps}
\label{sre}
In the first part of this section, we follow \cite{bbh} and \cite{lefterradulescu}.

We fix $k\in \N$ and a collection $\overline{d}=(d_1,\ldots,d_k)\in\Z^k$,   and we let $d:=d_1+\cdots+d_k$. The bounded domain $\Omega\subset\R^2$ is assumed to be simply connected and $C^{1,\beta}$.

We consider a collection of mutually distinct  points in $\O$, $a=(a_1,\ldots,a_k)\in\O^k_*$ (the prescribed singularities), and also a boundary datum $g\in H^{1/2}(\po;\so)$, of degree $d$. We denote by $\mathcal B_d$ the space of all such boundary data. Thus
\bes
{\mathcal B}_d:=\{ g\in H^{1/2}(\po ; \so) ;\, \deg(g,\po)=d\}.
\ees
For small $\rho>0$, we define the open set $\O_\rho=\O\setminus\bigcup_{j=1}^k \overline\D (a_j,\rho)$, 
and the classes of functions
\begin{gather}
\label{e2.1}
\mathcal{F}_{\rho,g}  =\left\lbrace v\in H^1(\O_\rho;\so);\; \tr v = g, \; \deg(v, C(a_j,\rho))=d_j \right\rbrace,
\\
\label{e2.2}
\widehat{\mathcal{F}}_\rho  = \left\lbrace v\in H^1(\O_\rho;\so);\; \deg(v,\po)=d, \; \deg(v, C(a_j,\rho))=d_j \right\rbrace.
\end{gather}
The functions in these classes have prescribed winding number $d_j$ around each $a_j$, and  prescribed boundary condition $g$ (respectively prescribed degree $d$) on $\po$. Of course, although we do not make this dependence explicit, the above classes  depend not only on $\rho$ and $g$, but also on $a$.

In \cite{bbh} and \cite{lefterradulescu}, minimization of the Dirichlet energy 
$1/2\int |\nabla v|^2$
over these spaces is studied, and  the following asymptotic expansions are obtained as $\rho\to 0$:
\begin{gather}
\label{e2.3}
\inf \left\lbrace\frac{1}{2}\int_{\O_\rho} \left|\nabla v\right|^2;\, v\in{\mathcal{F}}_{\rho,g}\right\rbrace  =\pi \left(\sum_{j=1}^k d_j^2\right)\log \frac{1}{\rho}+W(a,g)+O(\rho),\\
\label{e2.4}
\inf \left\lbrace\frac{1}{2}\int_{\O_\rho} \left|\nabla v\right|^2;\, v\in\widehat{\mathcal{F}}_\rho\right\rbrace =\pi \left(\sum_{j=1}^k d_j^2\right)\log \frac{1}{\rho}+\widehat{W}(a)+O(\rho).
\end{gather}
In the above expressions,  $W(a,g)$ and $\widehat{W}(a)$ are the so-called renormalized energies. These quantities depend not only on $a$ and $g$, but also  on $\overline{d}$ and $\O$. 

Explicit formulae for the above renormalized energies can be found in \cite{bbh} and \cite{lefterradulescu}, and involve the functions   $\Phi_{a,g}$ and $\widehat{\Phi}_a$ defined as follows. $\Phi_{a,g}$ is the unique solution of
\begin{equation}
\label{e2.5}
\begin{cases}
\Delta \Phi_{a,g} =  2\pi\sum_{j=1}^k d_j\delta_{a_j} & \quad \text{in }\O\\
\d\frac{\partial \Phi_{a,g}}{\partial \nu}  =  g \wedge  \frac{\p g}{\p\tau} & \quad \text{on }\po\\
\d\int_{\po} \Phi_{a,g} = 0 & \end{cases},\end{equation}
while $\widehat{\Phi}_a$ is the unique solution of 
\begin{equation}
\label{e2.6}
\begin{cases}
\Delta \widehat{\Phi}_a  =  2\pi\sum_{j=1}^k d_j\delta_{a_j} & \quad \text{in }\O\\
\widehat{\Phi}_a  = 0 & \quad \text{on }\po \end{cases}.\end{equation}
For further use, let us note that, if $\alpha\in\so$, then $\Phi_{a,g}=\Phi_{a,\alpha g}$. Therefore, we may naturally define $\Phi_{a,g}$ when $g$ is an equivalence class in $H^{1/2}(\po ; \so)/\so$.

We also define the regular parts $R_{a,g}$ and $\widehat{R}_a$ of $\Phi_{a,g}$ and $\widehat{\Phi}_a$ as follows:
\be
\label{e2.7}
R_{a,g}(x)  = \Phi_{a,g}(x)-\sum_{j=1}^k d_j \log |x-a_j|,\quad\fo x\in\O,
\ee
respectively
\be
\label{e2.8}
\widehat{R}_a(x)  = \widehat{\Phi}_a(x)-\sum_{j=1}^k d_j \log |x-a_j|, \quad\fo x\in\O.
\ee
The expressions of $W$ and $\widehat{W}$ are
\be
\label{wag}
W(a,g)  = -\pi \sum_{j\neq l} d_jd_l \log |a_j-a_l| + \frac{1}{2}\int_{\po} \Phi_{a,g}\left(g\wedge \frac{\p g}{\p\tau}\right) - \pi \sum_{j=1}^k d_j R_{a,g}(a_j),
\ee
respectively
\be
\label{e2.10}
\widehat{W}(a)  = -\pi \sum_{j\neq l} d_jd_l \log |a_j-a_l| - \pi \sum_{j=1}^k d_j \widehat{R}_a(a_j). 
\ee
The next result was proved in  \cite{lefterradulescu}.
\bpr\label{link}
We have
\begin{equation}
\label{e2.11}
\widehat{W}(a)=\inf\ \{  W(a,g);\, g\in {\cal B}_d\},
\end{equation}
and the infimum is attained in \eqref{e2.11}.
\epr
\noindent
Recall that ${\cal B}_d:=\{ g\in H^{1/2}(\po ; \so) ;\, \deg(g,\po)=d\}$. 

We present here an alternative proof of Proposition \ref{link}, in the course of which we exhibit a formula of the form
\[W(a,g)=\widehat{W}(a)+\text{ non negative terms },\]
which will be useful in the sequel.
\begin{proof}[Proof of Proposition \ref{link}]
We identify a map $\psi\in H^{1/2}(\po;\R)$ with its harmonic extension to $\O$, still denoted $\psi$. 
Given $\psi\in H^{1/2}(\po;\R)$, we define its (normalized) harmonic conjugate $\psi^*\in H^{1/2}(\po;\R)$ as follows. The harmonic extension of $\psi^*$ (still denoted $\psi^\ast$) is the unique solution of
\begin{equation}
\begin{cases}
\psi+\im\psi^*\;\text{is holomorphic in }\O,\\
\d\int_{\po}\psi^* = 0.\end{cases}
\end{equation}

Note that the Cauchy-Riemann equations imply
\begin{equation}
\label{e2.13}
\frac{\partial\psi^*}{\partial\nu}=-\frac{\partial \psi}{\partial\tau}\quad\mathrm{and}\quad\frac{\partial\psi^*}{\partial\tau}=\frac{\partial \psi}{\partial\nu},
\end{equation}
at least when $\psi$ is smooth. When $\psi$ is merely $H^{1/2}$, the distributions $\d\frac {\p\psi}{\p\nu}, \frac {\p\psi^*}{\p\nu}\in H^{-1/2}$ are respectively defined as  the trace on $\po$ of the normal derivatives of $\psi$ and $\psi^*$, and the equalities in \eqref{e2.13} are to be understood as equalities of distributions in $H^{-1/2}$.

We consider the space
\be
\label{e2.15}
H^{1/2}(\po;\R)/\R \simeq \dot{H}^{1/2}(\po;\R) := \left\lbrace \psi\in H^{1/2}(\po;\R);\,\int_{\po}\psi = 0\right\rbrace,
\ee
which is endowed with the natural norm
\be
\label{H1/2}
\left|\psi\right|^2_{H^{1/2}}  = \int_{\O}\left|\nabla\psi\right|^2  
 = \int_{\O}\left|\nabla\psi^*\right|^2   = \int_{\po} \frac{\partial\psi^*}{\partial\nu}\psi^* = -\int_{\po} \frac{\partial\psi}{\partial\tau}\psi^*.
\ee
If $\psi$ not smooth, then the two last integrals are to be understood  as $H^{-1/2}-H^{1/2}$ duality brackets.

Given $a\in \Omega_\ast^k$, we define the canonical boundary datum associated with $a$ as the unique element $g=g^a\in H^{1/2}(\po;\so)/\so$ such that $\deg(g,\po)=d$ and
\be
\label{e2.16}
g^a\wedge \frac{\p g^a}{\p\tau} = \frac{\partial \widehat{\Phi}_a}{\partial\nu}.
\ee 
Our first observation is that $g^a$ is well-defined and smooth. (It would be more accurate to assert that every map in the equivalence class defining $g^a$ is smooth.) Indeed, existence of a smooth $g:\po\to\so$ satisfying 
$\d g\wedge \frac{\p g}{\p\tau}=h$ (with given $h$) is equivalent to $h$ smooth and 
\be
\label{a1}\int_{\po}h=2\pi d. 
\ee
In addition, $g$ (if it exists) is unique modulo $\so$. In our case, we have 
$\d
 h=\frac{\partial \widehat{\Phi}_a}{\partial\nu}$,  
 which is smooth (since  $\widehat{\Phi}_a$ is smooth near $\po$). In addition, using the equation \eqref{e2.6}, we see that \eqref{a1} holds.
If we compare the definition of $g^a$ to the one of $\Phi_{a,g}$, we  see that the canonical datum $g^a$ is the unique $g$ (modulo multiplication by a constant in $\so$) such that 
\be
\label{e2.18}
\widehat{\Phi}_a=\Phi_{a,g}.
\ee
Given $g\in H^{1/2}(\po;\so)$ with $\deg(g,\po)=d$, we have $\deg(g/g^a,\po)=0$. Therefore, we may find $\psi=\psi_{a,g}\in H^{1/2}(\po;\R)$, unique modulo a constant, such that \cite{lss}
\be
\label{e2.19}
g=g^ae^{\im\psi}=g^ae^{\im\psi_{a,g}}.
\ee
Thus we have
\be
\label{e2.20}
\begin{cases}
\d\Delta\left[\Phi_{a,g}-\widehat{\Phi}_a\right]  =  0 &\text{in }\O \\
\d\frac{\partial}{\partial\nu}\left[\Phi_{a,g}-\widehat{\Phi}_a \right]  =  \frac{\p\psi}{\p\tau} &\text{on }\po\\
\d\int_{\po}\left(\Phi_{a,g}-\widehat{\Phi}_a\right)  = 0 &
\end{cases}.
\ee
Combining the above with the definition of the harmonic conjugate, we find that
\be
\label{e2.21}
\Phi_{a,g} = \widehat{\Phi}_a-\psi^*=\widehat{\Phi}_a-\psi_{a,g}^*.
\ee
Plugging \eqref{e2.21} into the expression of $W(a,g)$ given by formula \eqref{wag}, we find
\be
\label{e2.22}
\begin{aligned}
W(a,g)  =& -\pi \sum_{j\neq l } d_jd_l \log |a_j-a_l|  + \frac{1}{2}\int_{\po} \left(\widehat{\Phi}_a-\psi^*\right)\left(g^a\wedge \frac{\p g^a}{\p\tau}+\frac{\p\psi}{\p\tau}\right)\\
& - \pi \sum_{j=1}^k d_j (\widehat{R}_a(a_j)-\psi^*(a_j))\\
= &\widehat{W}(a) -\frac{1}{2}\int_{\po}\psi^*\left(g^a\wedge \frac{\p g^a}{\p\tau}\right)-\frac{1}{2}\int_{\po}\psi^*\frac{\p\psi}{\p\tau} +\pi \sum_{j=1}^k d_j\psi^*(a_j).
\end{aligned}
\ee
In the last equality we used the fact that $\widehat{\Phi}_a=0$ on $\po$. Furthermore, using the definition of $g^a$ and the fact that $\psi^*$ is harmonic, we obtain
\be
\label{e2.23}
\begin{aligned}
\int_{\po} \psi^*\left(g^a\wedge \frac{\p g^a}{\p\tau}\right)  =& \int_{\po} \psi^* \frac{\partial\widehat{\Phi}_a}{\partial\nu}\\
=& 
  \int_{\po}\frac{\partial\psi^*}{\partial\nu}\widehat{\Phi}_a +\int_{\O}\psi^*\Delta\widehat{\Phi}_a 
= 2\pi\sum_{j=1}^k d_j \psi^*(a_j).
\end{aligned}
\ee
Using \eqref{H1/2}, \eqref{e2.22} and \eqref{e2.23}, we finally obtain
\be
\label{e2.24}
W(a,g)=\widehat{W}(a)+\frac{1}{2}\left|\psi_{a,g}\right|^2_{H^{1/2}}.
\ee
In particular, we recover the conclusion of Proposition \ref{link} in the following stronger form: the minimum of $W(a, \cdot)$ is attained (exactly) when $g=g^a$ (modulo $\so$).
\end{proof}

\begin{remark}
The canonical boundary datum $g^a$ will play a crucial role in our subsequent analysis. We emphasize here the fact that $g^a$ is the (unique modulo $\so$) solution of 
\be
\l{aba1}
g^a \wedge \frac{\partial g^a}{\partial \tau} = \frac{\partial\widehat{\Phi}_a}{\partial\nu}\quad\text{on }\partial\Omega.
\ee
\end{remark}

The limit (as $\rho\to 0$) of the variational problem \eqref{e2.3} is also connected to the so-called canonical harmonic map $u_{*,a,g}$ associated to prescribed singularities $a\in\Omega_*^k$ and to the Dirichlet condition $g\in H^{1/2}(\partial\Omega ; \mathbb{S}^1)$. 
In fact,  in \cite[Chapter I]{bbh} it is proved that the unique solution $u_{\rho,g}$ of the minimization problem $\inf \,\left\lbrace \int |\nabla u|^2 \, ; \, u\in\mathcal{F}_{\rho,g} \right\rbrace $ tends to $u_{*,a,g}$,  in $C^k_{loc}(\Omega\setminus \{ a_j \} )$ as $\rho\to 0$.\footnote{Actually, in \cite[Chapter I]{bbh} the map $g$ is supposed smooth, but the argument adapts to a general $g\in H^{1/2}(\po ; \so)$.} 

The canonical harmonic map is defined by the formula
\begin{equation}\label{u*1}
\begin{cases}
\d u=u_{*,a,g}=e^{\im H}\prod_{j=1}^k \left( \frac{z-a_j}{|z-a_j|} \right) ^{d_j} &\text{in }\O\\
\Delta H=0&\text{in }\O\\
u=g&\text{on }\po
\end{cases}.
\end{equation}
The fact that $\deg (g,\partial\Omega)=d=\sum d_j$ guarantees that $H=H_g$ is well defined. Indeed, there exists $\psi\in H^{1/2}(\partial\Omega ; \R)$ such that
\[
g\prod_{j=1}^k \left( \frac{z-a_j}{|z-a_j|} \right) ^{-d_j} = e^{\im \psi},
\]
and then we can simply let $H$ be the harmonic extension of $\psi$. On the other hand, we note that $H$ is uniquely defined up to a multiple of $2\pi$. 

Equivalently, $u$ in \eqref{u*1} is characterized  by \cite[Chapter I]{bbh}
\begin{equation}\label{u*2}
\begin{cases}
|u|=1\\
\d u\wedge \frac{\partial u}{\partial{x_1}}  = -\frac{\partial \Phi_{a,g}}{\partial{x_2}}\\
\d u\wedge \frac{\partial u}{\partial{x_2}} = \frac{\partial \Phi_{a,g}}{\partial{x_1}}\\
u=g\text{ on }\po
\end{cases}.
\end{equation}
In particular, we have
\be
\l{e21121}
u_{\ast, a, g}\wedge\frac{\p u_{\ast, a, g}}{\p\nu}=-\frac{\p\Phi_{a,g}}{\p\tau}\quad\text{on }\po
\ee
and
\be
\l{e21122}
\int_{\po}u_{\ast, a, g}\wedge\frac{\p u_{\ast, a, g}}{\p\nu}=0.
\ee
\begin{remark}
For the minimization problem \eqref{e2.4}, the situation is similar. As established  in \cite{lefterradulescu},
 the solution $v_\rho$ to $\inf\,\left\lbrace \int |\nabla v|^2\, ;\, v\in\widehat{\mathcal F}_\rho\right\rbrace$ converges (in an appropriate sense) as $\rho\to 0$, to $v_{*,a}:=u_{*,a,g^a}$. 
%
 Since $g^a$ is defined modulo $\mathbb{S}^1$, $v_{*,a}$ is also defined modulo $\mathbb{S}^1$. Therefore, in this context  the convergence actually means that subsequences of $(v_\rho)$ converge to representatives (modulo $\mathbb{S}^1$) of $v_{*,a}$. 
\end{remark}

We end this section with  the definition of the following quantity, which will play a very important  role in what follows.  For $a\in\Omega_*^k$ and $g\in H^{1/2}(\partial\Omega ; \mathbb{S}^1)$, we set
\begin{equation}\label{defN(a,g)}
N(a,g):=u_{*,a,g}\wedge \frac{\partial u_{*,a,g}}{\partial\nu} =-\frac{\p\Phi_{a,g}}{\p\tau} \in H^{-1/2}(\partial\Omega ; \R).
\end{equation}

\section{Transport of formulas onto the unit disc}
\label{st}

Let  $f:\mathbb{D}\to\Omega$ be a conformal representation. The assumption $\Omega\in C^{1,\beta}$ ensures that $f$ and its inverse $\varphi:=f^{-1}:\Omega\to\mathbb{D}$ are $C^{1,\beta}$ up to the boundary.

The goal of this section is to understand how the objects defined in Section~\ref{sre} are transported by $\varphi$ and $f$. 

We will stress the dependence on the domain by using superscripts (e.g. $W=W^{\Omega}$). For $\alpha\in\mathbb{D}_*^k$, the notation $a=f(\alpha)$ stands for $a:=(f(\alpha_1),\ldots,f(\alpha_k))\in \Omega_*^k$.

First of all, for $a\in\Omega_*^k$, we have
\begin{equation}
\label{transformphiag}
\Phi_{a,g}^\Omega = \Phi^{\mathbb{D}}_{\varphi(a),g\circ f} \circ \varphi+ C,
\end{equation}
where $C=C(a,g,f)=-\int_{\so}\Phi^{\mathbb{D}}_{\varphi(a),g\circ f}|f'|$. Indeed, \eqref{transformphiag} is justified as follows. By a direct calculation, both sides of \eqref{transformphiag}  satisfy the same Poisson equation, with the same Neumann boundary condition. The constant $C$ comes from the normalization condition $\int_{\partial\Omega}\Phi_{a,g}=0$. The same argument applies to show that
\begin{equation} 
\label{transformphihata}
\widehat{\Phi}_a^\Omega = \widehat{\Phi}^ {\mathbb{D}}_{\varphi(a)}\circ \varphi.
\end{equation}
Here there is no renormalization constant since $\widehat{\Phi}_a$ satisfies a Dirichlet boundary condition.

Normal and tangential derivatives transform in the following way. If $v:\mathbb{D}\to\C$, then
\begin{align}
\label{transformtangentderiv}
\frac{\partial}{\partial\tau}\left[ v\circ \varphi\right] (z) & = |\varphi'(z)|\frac{\partial v}{\partial\tau}\left(\varphi(z)\right),\quad z\in\partial\Omega, \\
\label{transformnormalderiv}
\frac{\partial}{\partial\nu}\left[ v\circ \varphi\right] (z) & = |\varphi'(z)|\frac{\partial v}{\partial\nu}\left(\varphi(z)\right),\quad z\in\partial\Omega. 
\end{align}

Using \eqref{transformphihata}, \eqref{transformtangentderiv}, \eqref{transformnormalderiv} together with formula \eqref{e2.16} characterizing $g^a$, we find, for $a\in\Omega_*^k$,
\begin{equation}\label{transformga}
g^a \circ f = g^{\varphi(a)}.
\end{equation}

On the other hand, we claim that
\begin{equation}\label{transformu*}
u^{\Omega}_{*,a,g} = u^{\mathbb{D}}_{*,\varphi(a),g\circ f} \circ \varphi.
\end{equation}
Indeed, this follows from the observation that the two sides of \eqref{transformu*} agree on $\po$, combined with \eqref{u*1} and with the fact, when $H$ is harmonic in $\D$,  we may write 
\bes
\frac{\va (z)-\va(a)}{|\va (z)-\va (a)|}e^{\im H\circ \va(z)}=\frac{z-a}{| z-a|}e^{\im K(z)},\quad\text{with }K\text{ harmonic in }\O.
\ees
%
As a consequence of \eqref{transformu*} and \eqref{transformnormalderiv}, we obtain, recalling  the definition \eqref{defN(a,g)} of $N$,
\begin{equation}\label{transformN}
N^{\Omega}(a,g) = |\varphi'|\, N^{\mathbb D}(\varphi(a),g\circ f) \circ \varphi.
\end{equation}

The formulas of the renormalized energies $\widehat{W}$ and $W$ transport in a more complicated way.
\begin{lemma}\label{transformW}
Let $\alpha\in\mathbb{D}_*^k$, $a:=f(\alpha)$ and $g:\partial\Omega\to\mathbb{S}^1$. Then
\begin{gather}
\label{transformW(a,g)}
W^{\Omega}(a,g)  = W^{\mathbb D}(\alpha,g\circ f)+\pi \sum_j d_j^2 \log |f'(\alpha_j)|,
\\
\label{transformhatW(a)}
\widehat{W}^{\Omega}(a)  = \widehat{W}^{\mathbb D}(\alpha)+\pi \sum_j d_j^2 \log |f'(\alpha_j)|.
\end{gather}
\end{lemma}
\begin{proof}
Using definition of $R_{a,g}$ \eqref{e2.7},  together with \eqref{transformphiag}, we compute, for $z\in\mathbb{D}$,
\begin{equation}\label{transformRag}
\begin{split}
R^{\Omega}_{a,g}(f(z)) & = \Phi^{\mathbb D}_{\alpha,g\circ f}(z)-\sum_{l=1}^k d_l \log | f(z) - f(\alpha_l) | + C \\
& = R^{\mathbb D}_{\alpha,g\circ f}(z) 
-\sum_{l=1}^k d_l \log \left| \frac{f(z)-f(\alpha_l)}{z-\alpha_l} \right| +C.
\end{split}\end{equation}
The above is well-defined when $z\neq\alpha_j$, and extends by continuity at $z=\alpha_j$. 
In particular, 
\begin{equation}\label{transformRag2}
R^{\Omega}_{a,g}(f(\alpha_j)) = R^{\mathbb D}_{\alpha,g\circ f}(\alpha_j) 
-\sum_{l\neq j} d_l \log \left| \frac{f(\alpha_j)-f(\alpha_l)}{\alpha_j-\alpha_l} \right| - d_j \log |f'(\alpha_j)| +C.
\end{equation}
Finally, we plug \eqref{transformphiag} and \eqref{transformRag2} into formula \eqref{wag} expressing $W$ in terms of $\Phi_{a,g}$ and $R_{a,g}$. We obtain, using also the fact that $\deg (g,\partial\Omega) = d = \sum d_j$,
\bes
\begin{aligned}
W^{\Omega}(a,g) =&   -\pi \sum_{j\neq l} d_jd_l \log |\alpha_j-\alpha_l| 
+ \frac{1}{2}\int_{\po} \Phi^{\mathbb D}_{\alpha,g\circ f} \circ \varphi \left(g\wedge \frac{\p g}{\p\tau}\right) \\
&+\frac{1}{2} C \int_{\partial\Omega} g\wedge\frac{\partial g}{\partial\tau} - \pi\sum_j d_j C 
- \pi \sum_{j=1}^k d_j R^{\mathbb D}_{\alpha,g\circ f}(\alpha_j)
+\pi\sum_{j=1}^k d_j^2 \log |f'(\alpha_j)| \\
=&  -\pi \sum_{j\neq l} d_jd_l \log |\alpha_j-\alpha_l| 
+ \frac{1}{2}\int_{\p\mathbb{D}} \Phi^{\mathbb D}_{\alpha,g\circ f} (g\circ f)\wedge \frac{\p}{\p\tau} (g\circ f)\\
& \quad
- \pi \sum_{j=1}^k d_j R^{\mathbb D}_{\alpha,g\circ f}(\alpha_j)
+\pi\sum_{j=1}^k d_j^2 \log |f'(\alpha_j)|\\
& = W^{\mathbb D}(\alpha,g\circ f)+\pi \sum_j d_j^2 \log |f'(\alpha_j)|.
\end{aligned}
\ees
Formula \eqref{transformhatW(a)} can be proved following the same lines (the calculations are even simpler than for \eqref{transformW(a,g)}).  Alternatively, we can obtain \eqref{transformhatW(a)} via the relation $\widehat{W}(a)=W(a,g^a)$.
\end{proof}
%

\section{Explicit formulas in the unit disc}
\label{sef}

In this section we derive explicit formulas for $\widehat{W}^{\mathbb D}$, $W^{\mathbb{D}}$ and $N^{\mathbb{D}}$.

We start by recalling the explicit formulas for $\widehat\Phi^\D_\alpha$ and $\widehat W^\D$ \cite{lefterradulescu}:  for $\alpha\in \mathbb{D}_*^k$, we have
\begin{align}
\label{hatphidisc}
\widehat{\Phi}_\alpha^{\mathbb D} (z) & = \sum_{j=1}^k d_j \left( \log|z-\alpha_j|-\log|1-\overline{\alpha}_jz| \right), \quad\forall\,  z\in\mathbb{D},
\\
\label{hatwdisc}
\widehat{W}^{\mathbb D}(\alpha) & =  -\pi\sum_{j\neq l} d_jd_l\log |\alpha_j-\alpha_l|+\pi\sum_{j\neq l}d_jd_l\log |1-\overline{\alpha_j}\alpha_l|\\ \nonumber
& \quad  +\pi\sum_j d_j^2\log(1-|\alpha_j|^2).
\end{align}
The formulas for 
 $W$ and $N$ are more involved.
\begin{lemma}\label{discexplicit}
Let $\alpha^0 \in \mathbb{D}_*^k$ be fixed, and $g^0:=g^{\alpha^0}:\so\to\mathbb{S}^1$ be an associated canonical boundary map.  (Recall that $g^{0}$ is defined up to a multiplicative constant.) Then it holds: \\
(i) For $\alpha\in\mathbb{D}_*^k$ and for $\psi\in H^{1/2}(\so;\R)$,
\begin{equation}\label{Wdisc}
W^{\mathbb D}(\alpha,g^0e^{\im \psi})  = \widehat{W}^{\mathbb D}(\alpha)+\frac{1}{2}\int_{\mathbb D}\left|\nabla \left(\psi_{\alpha,g^0}^*+\psi^*\right)\right|^2 
,
\end{equation}
and, for $z\in\mathbb{D}$,
\begin{equation}\label{nablapsi*ag0disc}
\nabla\psi_{\alpha,g^0}^* (z) = 
2\sum_{j=1}^k d_j \left(\frac{\alpha_j(1-\overline{\alpha}_j z)}{|1-\overline{\alpha}_j z|^2}-\frac{\alpha^0_j(1-\overline{\alpha}^0_j z)}{|1-\overline{\alpha}^0_j z|^2}\right) \in \C\simeq \R^2.
\end{equation}
(ii) For $\alpha\in\mathbb{D}_*^k$ and for $\psi\in H^{1/2}(\so;\R)$,
\begin{equation}\label{Ndisc}
N^{\mathbb D}(\alpha,g^0 e^{\im\psi}) = \frac{\partial\psi^*}{\partial\tau} + 2 \sum_j d_j \frac{\alpha_j^0\wedge z}{|z-\alpha_j^0|^2} - 2\sum_j d_j \frac{\alpha_j\wedge z}{|z-\alpha_j|^2}.
\end{equation}
\end{lemma}
\begin{proof}[Proof of {\it (i)}]
Since we will always work in the unit disc, we drop the superscript $\mathbb{D}$.

We know from (\ref{e2.24}) that for $g\in H^{1/2}(\so ;\so)$,
\be\label{e3.27}
W(\alpha,g)=\widehat{W}(\alpha)+\frac{1}{2}\left|\psi_{\alpha,g}\right|^2_{H^{1/2}},
\ee
where $\psi_{\alpha,g}$ is defined (modulo a constant) in (\ref{e2.19}) by
\be\label{e3.28}
g=g^\alpha e^{\im \psi_{\alpha,g}}.
\ee
Taking $g=g^0e^{\im \psi}$, and using $g^0=g^\alpha e^{\im \psi_{\alpha,g^0}}$, we find
\be\label{e3.29}
g=g^\alpha e^{\im \psi_{\alpha,g^0}}e^{\im \psi}=g^\alpha e^{\im (\psi_{\alpha,g^0}+\psi)},
\ee
so that it holds
\be\label{e3.30}
\psi_{\alpha,g}=\psi_{\alpha,g^0}+\psi.
\ee
This leads to 
\be\label{e3.31}
W(\alpha,g^0e^{\im \psi})  =\widehat{W}(\alpha)+\frac{1}{2}\left|\psi_{\alpha,g^0}+\psi\right|^2_{H^{1/2}} =\widehat{W}(\alpha)+\frac{1}{2}\int_{\mathbb D}\left|\nabla\left(\psi_{\alpha,g^0}^*+\psi^*\right)\right|^2,
\ee
i.e., \eqref{Wdisc} holds. 
In order to complete the proof of {\it (i)}, it remains to compute $\nabla \psi^*_{\alpha,g^0}$.

Recall that $\psi_{\alpha,g^0}^*$ is characterized by
\be\label{e3.34}
\begin{cases}
\Delta \psi_{\alpha,g^0}^*=0 & \text{in }\mathbb{D}\\
\d\frac{\partial}{\partial\nu}\psi_{\alpha,g^0}^*=-\frac{\partial}{\partial\tau}\psi_{\alpha,g^0} & \text{on }\so\\
\d\int_{\so}\psi_{\alpha,g^0}^* =0 &
\end{cases}.
\ee
Since $e^{\im \psi_{\alpha,g^0}}=g^0/g^{\alpha}$, we have
\be\label{e3.35}
\frac{\partial\psi_{\alpha,g^0}}{\partial\tau} =g^0\wedge \frac {\p g^0}{\p\tau}- g^\alpha\wedge \frac{\p g^\alpha}{\p\tau}.
\ee
By definition of $g^\alpha$ and $g^0=g^{\alpha^0}$, and using (\ref{hatphidisc}), we obtain
\be\label{e3.36}
\begin{split}
g^\alpha\wedge \frac{\p g^\alpha}{\p\tau} & = \frac{\partial\widehat{\Phi}_\alpha}{\partial\nu} = \sum_j d_j \frac{\partial}{\partial\nu}\left[\log|z-\alpha_j|-\log|1-\overline{\alpha_j}z|\right],\\
g^0\wedge \frac{\p g^0}{\p\tau} & = \frac{\partial\widehat{\Phi}_{\alpha^0}}{\partial\nu} = \sum_j d_j \frac{\partial}{\partial\nu}\left[\log|z-\alpha^0_j|-\log|1-\overline{\alpha^0_j}z|\right].
\end{split}\ee

We also note the identity
\be\label{e3.38}
1 = \frac{\partial}{\partial\nu}\left[\log|1-\overline{\alpha}z|+\log|z-\alpha|\right],\quad \fo \alpha\in\mathbb{D}.
\ee

Combining \eqref{e3.35})-\eqref{e3.38},  
we obtain
\be\label{e3.39}
\frac{\partial\psi_{\alpha,g^0}^*}{\partial\nu}=-\frac{\partial\psi_{\alpha,g^0}}{\partial\tau}=
\frac{\partial}{\partial\nu}\left[2\sum_j d_j\left(\log|1-\overline{\alpha^0_j}z|-\log|1-\overline{\alpha_j}z|\right)\right].
\ee
Therefore, there exists a constant $c(\alpha)\in\R$ such that
\be\label{e3.40}
\psi_{\alpha,g^0}^*(z)=2\sum_j d_j\left(\log|1-\overline{\alpha}^0_j z|-\log|1-\overline{\alpha}_j z|\right)+c(\alpha),\quad\quad\forall\,  x\in\mathbb{D}.
\ee
Indeed, the right-hand side of \eqref{e3.40} satisfies \eqref{e3.34}, and so does $\psi^*_{\alpha, g^0}$. The constant $c(\alpha)$ is determined by the normalization condition $\int \psi^*_{\alpha, g^0}=0$. From \eqref{e3.40} we immediately obtain \eqref{nablapsi*ag0disc}.
\end{proof}

\begin{proof}[Proof of $(ii)$]
In view of formula \eqref{u*1}, we have
\begin{equation}\label{formulaN1}
N(\alpha,g^0e^{\im \psi})  =\frac{\partial H}{\partial\nu} + 
\frac{\partial}{\partial\nu}\left[\sum_j d_j \theta(z-\alpha_j)\right] =\frac{\partial H^*}{\partial\tau} - \sum_j d_j \frac{\partial}{\partial\tau}\left[ \log |z-\alpha_j| \right],
\end{equation}
where $H^*$ is the harmonic conjugate of $H$, characterized (up to a constant) by
\begin{equation}\label{H*}
\begin{cases}
\Delta H^* = 0 & \text{in }\mathbb{D} \\\d
\frac{\partial H^*}{\partial\nu} = -\frac{\partial H}{\partial\tau} & \text{on }\so.
\end{cases}
\end{equation} 
On the boundary $\so$, we have
\begin{equation}\label{eiH}
e^{\im H} = \prod_j \left(\frac{z-\alpha_j}{|z-\alpha_j|}\right)^{-d_j} g = \prod_j \left(\frac{z-\alpha_j}{|z-\alpha_j|}\right)^{-d_j} g^0 e^{\im\psi}, 
\end{equation}
so that
\begin{equation}\label{tangentderivH1}
\begin{split}
\frac{\partial H}{\partial\tau} & = 
\frac{\partial\psi}{\partial\tau} + g^0 \wedge \frac{\partial g^0}{\partial \tau} 
- \sum_j d_j \frac{\partial}{\partial\tau}\left[\theta(z-\alpha_j)\right] \\
& = -\frac{\partial\psi^*}{\partial\nu} + \frac{\partial \widehat{\Phi}_{\alpha^0}}{\partial\nu}
- \sum_j d_j \frac{\partial}{\partial\nu}\left[ \log |z-\alpha_j|\right] \\
& = -\frac{\partial\psi^*}{\partial\nu} + \sum_j d_j \frac{\partial}{\partial\nu}\left[\log|z-\alpha_j^0|-\log|1-\overline{\alpha}_j^0 z|\right]  - \sum_j d_j \frac{\partial}{\partial\nu}\left[ \log |z-\alpha_j|.\right]
\end{split}\end{equation}
Here we have used the definition of $g^0=g^{\alpha^0}$ and the explicit formula \eqref{hatphidisc} for $\widehat{\Phi}_\alpha$. Using \eqref{e3.38}, 
we obtain
\begin{equation}\label{tangentderivH2}
\frac{\partial H}{\partial\tau} = -\frac{\partial}{\partial\nu}\left[ \psi^* + 
\sum_j d_j (2\log|1-\overline{\alpha}_j^0 z|-\log|1-\overline{\alpha}_j z|)\right].
\end{equation}
We deduce that there exists a constant $c=c(\psi,\alpha)$ such that
\begin{equation}\label{formulaH*}
H^* = \psi^* + 
\sum_j d_j (2\log|1-\overline{\alpha}_j^0 z|-\log|1-\overline{\alpha}_j z|)+c.
\end{equation}
From \eqref{formulaH*} and \eqref{formulaN1} we obtain
\begin{equation}\label{formulaT*2}
N(\alpha,g^0 e^{\im\psi}) = \frac{\partial\psi^*}{\partial\tau} + 
\sum_j d_j \frac{\partial}{\partial\tau}\left[
2\log|1-\overline{\alpha}_j^0 z|-\log|1-\overline{\alpha}_j z| - \log|z-\alpha_j| \right].
\end{equation}
Using the fact that for every $\alpha\in\mathbb{D}$ we have
\begin{equation}\label{tangentderivlog}
\frac{\partial}{\partial\tau}\left[\log|z-\alpha|\right]
=\frac{\partial}{\partial\tau}\left[\log|1-\overline{\alpha}z|\right]
=\frac{\alpha\wedge z}{|z-\alpha|^2},\quad \fo z\in\so, 
\end{equation}
we finally obtain
\begin{equation}\label{formulaN2}
N(\alpha,g^0 e^{\im\psi}) = \frac{\partial\psi^*}{\partial\tau} + 2 \sum_j d_j \frac{\alpha_j^0\wedge z}{|z-\alpha_j^0|^2} - 2\sum_j d_j \frac{\alpha_j\wedge z}{|z-\alpha_j|^2},
\end{equation}
as claimed.
\end{proof}

%
%
%

\section{Nondegeneracy of $W$ is stable}\label{snd}

In this section we show that, if $a^0\in(\Omega_0)_*^k$ is a nondegenerate critical point of $W^{\O_0}(\cdot,g_0)$, with $g_0:\p\O_0\to\so$,  then for  $\Omega$ \enquote{close to}  $\Omega_0$, and for  $g:\po\to\so$ \enquote{close to} $g_0$, there exists a unique nondegenerate critical point $a$ of $W^\O(\cdot,g)$ \enquote{close to} $a^0$. Unlike the analysis we perform in subsequent sections, smoothness (of the domain or of the boundary datum) is not crucial here. In order to emphasize this fact, we first state and prove a result concerning rough boundary datum (Proposition \ref{propndstable}). We next present a \enquote{smoother} variant of the stability result (Proposition \ref{p301}).

The notion of closeness will be expressed in terms of  conformal representations. Let us first introduce some definitions.
Let $X$ be the space
\be\label{e3.22}
X:=\left\lbrace f\in C^1\left(\overline{\mathbb{D}};\C\right);\; f \text{ is holomorphic in }\mathbb{D}\right\rbrace,
\ee
which is a Banach space with 
 the  $\|\cdot\|_{C^1}$ norm. In $X$ we will consider the open set
\be\label{e3.23}
V:=\left\lbrace f\in X;\, f\text{ is bijective and } f^{-1}\in X \right\rbrace.
\ee
Every $f\in V$ induces a conformal representation $f:\mathbb{D}\to\O:=f(\mathbb{D})$, which is $C^1$ up to the boundary. In what follows, we denote by $f^{-1}$ both the inverse of $f:\overline\D\to\overline\O$ and of $\d f_{|\so}:\so\to\po$. 

Similar considerations apply to the space
\be\label{e3.22a}
X_\beta:=\left\lbrace f\in C^{1,\beta}\left(\overline{\mathbb{D}};\C\right);\; f \text{ is holomorphic in }\mathbb{D}\right\rbrace,
\ee
and to the open set
\be\label{e3.23a}
V_\beta:=\left\lbrace f\in X;\, f\text{ is bijective and } f^{-1}\in X_\beta \right\rbrace.
\ee
Here, $0<\beta<1$.
\begin{proposition}\label{propndstable}
Let $\Omega_0$ be a smooth bounded simply connected $C^{1,\beta}$ domain and $f_0:\mathbb{D}\to\Omega_0$ be a conformal representation. Assume that there exists $\alpha^0\in\mathbb{D}_*^k$ and $g_0\in H^{1/2}(\so;\mathbb{S}^1)$ such that $a^0:=f_0(\alpha^0)$ is a nondegenerate critical point of $W^{\Omega_0}(\cdot,g_0\circ f_0^{-1})$.

Then there exist a neighborhood $\mathcal{V}$ of $(f_0,0)$ in $V\times H^{1/2}(\so;\R)$, a smooth map $\alpha : \mathcal{V} \to \mathbb{D}_*^k$, and some $\delta>0$, such that the following  holds.

Let $(f,\psi) \in \mathcal{V}$ and consider the domain $\Omega:=f(\mathbb{D})$ together with the boundary datum $g:=(g_0 e^{\im\psi})\circ f^{-1} \in H^{1/2}(\partial\Omega;\mathbb{S}^1)$. Then $W^{\Omega}(\cdot,g)$ admits a unique critical point $a\in\Omega_*^k$ satisfying $|f^{-1}(a)-\alpha^0| <\delta$. This $a$ is  given by the map $a(f,\psi)=f\left(\alpha(f,\psi)\right)$. Furthermore, $a$ is a nondegenerate critical point of $W^\O(\cdot , g)$.
\end{proposition}
Before proving Proposition~\ref{propndstable} we state as a lemma the following smoothness result.

\begin{lemma}\label{wsmooth}
The map $\widetilde{W}:\mathbb{D}_*^k\times V \times H^{1/2}(\so;\R)\to\R$, defined by
\begin{equation}\label{tildew1}
\widetilde{W}(\alpha,f,\psi)=W^{f(\mathbb{D})}\left(f(\alpha),(g_0e^{\im \psi})\circ f^{-1}\right),
\end{equation}
is smooth.

Similarly, the map $\widetilde{W}^\beta:\mathbb{D}_*^k\times V_\beta\times C^{1,\beta}(\so;\R)\to\R$, defined by
\begin{equation}\label{barw}
\widetilde{W}^\beta(\alpha,f,\psi)=W^{f(\mathbb{D})}\left(f(\alpha),(g_0e^{\im \psi})\circ f^{-1}\right),
\end{equation}
is smooth.
\end{lemma}
\begin{proof}[Proof of Lemma \ref{wsmooth}]
The idea is to rely on the formulas derived in Sections \ref{st} and \ref{sef} in order to obtain an explicit formula for $\widetilde{W}$, from which it will be clear that $\widetilde{W}$ is smooth. 

To start with, formula \eqref{transformW(a,g)} gives
\begin{equation}\label{wsmooth1}
\widetilde{W}(\alpha,f,\psi) = W^{\mathbb{D}}(\alpha,g_0 e^{\im\psi})+\pi \sum_j d_j^2 \log |f'(\alpha_j)|.
\end{equation}

Using the fact that for holomorphic $f$ all derivatives can be estimated locally using only $\|f\|_\infty$, it can be easily shown that the maps
\begin{equation}\label{wsmooth2}
\mathbb{D}_*^k \times V \ni (\alpha,f) \mapsto \log |f'(\alpha_j)|
\end{equation}
are smooth. 

Therefore the second term in the right-hand side of \eqref{wsmooth1} is smooth, and in order to complete the proof of Lemma \ref{wsmooth} it suffices to prove that
\begin{equation}\label{wsmooth3}
\mathbb{D}_*^k \times H^{1/2}(\so;\R) \ni (\alpha,\psi)\mapsto W^{\mathbb{D}}(\alpha,g_0 e^{\im\psi}):=P_{g_0}(\alpha,\psi)
\end{equation}
is smooth. Clearly, if $g\in H^{1/2}(\so ; \so)$ is such that $\deg (g, \so)=\deg (g_0, \so)$, then we may write $g=g_0e^{\im\psi_0}$ for some $\psi_0\in H^{1/2}(\so ; \R)$, and then we have $P_g(\alpha,\psi)=P_{g_0}(\alpha,\psi+\psi_0)$. This implies that the smoothness of $P_{g_0}$ does not depend on the choice of $g_0$. Therefore, we may assume that 
 $g_0=g^{\alpha^0}$ for some $\alpha^0\in\mathbb{D}_*^k$.
  This assumption allows us to use Lemma~\ref{discexplicit}. Using \eqref{Wdisc}, we obtain\begin{equation}\label{wsmooth4}
\begin{split}
W^{\mathbb{D}}(\alpha,g_0 e^{\im\psi}) & = \widehat{W}^{\mathbb D}(\alpha)+\frac{1}{2}\int_{\mathbb D}\left|\nabla\psi_{\alpha,g_0}^*\right|^2 
+\frac{1}{2}\int_{\mathbb D}\left|\nabla\psi^*\right|^2\\
& \quad
+\int_{\mathbb D}\nabla\psi_{\alpha,g_0}^*\cdot\nabla\psi^*.
\end{split}
\end{equation}
We examine the smoothness of the four terms on the right-hand side of \eqref{wsmooth4}. The first term depends only on $\alpha$ and is smooth thanks to formula \eqref{hatwdisc}. The second term depends only on $\alpha$ and is smooth thanks to formula \eqref{nablapsi*ag0disc}. The third term depends only on $\psi$ and is a continuous quadratic form, hence it is smooth. The fourth and last term depends linearly on $\psi$ and is smooth thanks to formula \eqref{nablapsi*ag0disc} again.

Hence the map \eqref{wsmooth3} is smooth, and the proof of the $H^{1/2}$ part of the lemma is complete.

The proof of the $C^{1,\beta}$ part of the  follows the same lines and is left to the reader.
\end{proof}

\begin{proof}[Proof of Proposition~\ref{propndstable}]
Let us first remark the following fact. Fix $f\in V$ and $\psi\in H^{1/2}(\partial\mathbb{\D};\R)$ and consider the domain $\Omega=f(\mathbb{D})$ together with the boundary datum $g=(g_0 e^{\im\psi})\circ f^{-1}$. Then, for any $\alpha\in D_*^k$, $f(\alpha)$ is a nondegenerate critical point of $W^{\Omega}(\cdot,g)$ if and only if $\alpha$ is a nondegenerate critical point of $\widetilde{W}(\cdot,f,\psi)$. This is a simple consequence of the fact that $f$ induces a diffeomorphism from $\mathbb{D}_*^k$ into $\Omega_*^k$.

We consider the map $F:\mathbb{D}_*^k\times V \times H^{1/2}(\so;\R)\to\R^{2k}$,
\begin{equation}\label{propndstable1}
F: (\alpha,f,\psi)\mapsto \nabla_\alpha \widetilde{W}(\alpha,f,\psi).
\end{equation}
Lemma~\ref{wsmooth} ensures that $F$ is smooth. Moreover, the assumption that $a^0$ is a nondegenerate critical point of $W^{\Omega_0}(\cdot,g_0\circ f^{-1})$ ensures that $\alpha^0$ is a nondegenerate critical point of $\widetilde{W}(\cdot,f_0,0)$. Therefore $F(\alpha^0,f_0,0)=0$, and $D_\alpha F(\alpha^0,f_0,0)$ is invertible. 

This enables us to apply the implicit function theorem: there exist of a neighborhood $\mathcal{V}$ of $(f_0,0)$ in $V\times H^{1/2}(\so;\R)$, a smooth map $\alpha : \mathcal{V} \to \mathbb{D}_*^k$, and $\delta>0$, such that, for $(f,\psi)\in\mathcal{V}$ and $|\alpha-\alpha^0|<\delta$,
\begin{equation}\label{propndstable2}
F(\alpha,f,\psi)=0 \Longleftrightarrow \alpha=\alpha(f,\psi).
\end{equation} 
We may also assume that $D_{\alpha}F(\alpha(f,\psi),f,\psi)$ is invertible, so that $\alpha(f,\psi)$ is a nondegenerate critical point of $\widetilde{W}(\cdot,f,\psi)$. This implies that $a:=f\left(\alpha(f,\psi)\right)$ is a nondegenerate critical point of $W^{\Omega}(\cdot,g)$, where $\Omega=f(\mathbb{D})$ and $g=(g_0e^{\im\psi})\circ f^{-1}$. In view of \eqref{propndstable2}, $a$ is the unique critical point of $W^\O(\cdot, g)$ satisfying $|f^{-1}(a)-\alpha^0|<\delta$.

The proof of Proposition \ref{propndstable} is complete.
\end{proof}

%

In what follows, we will use the following smoother version of  Proposition \ref{propndstable}.
\begin{proposition}\label{p301}
Let $\Omega_0$ be a smooth bounded simply connected $C^{1,\beta}$ domain and $f_0:\mathbb{D}\to\Omega_0$ be a conformal representation. Assume that there exists $\alpha^0\in\mathbb{D}_*^k$ and $g_0\in C^{1, \beta}(\so;\mathbb{S}^1)$ such that $a^0:=f_0(\alpha^0)$ is a nondegenerate critical point of $W^{\Omega_0}(\cdot,g_0\circ f_0^{-1})$.

Then there exist a neighborhood $\mathcal{V}$ of $(f_0,0)$ in $V_\beta\times C^{1,\beta}(\so;\R)$, a smooth map $\alpha : \mathcal{V} \to \mathbb{D}_*^k$, and some $\delta>0$, such that the following  holds.

Let $(f,\psi) \in \mathcal{V}$ and consider the domain $\Omega:=f(\mathbb{D})$ together with the boundary datum $g:=(g_0 e^{\im\psi})\circ f^{-1} \in C^{1,\beta}(\partial\Omega;\mathbb{S}^1)$. Then $W^{\Omega}(\cdot,g)$ admits a unique critical point $a\in\Omega_*^k$ satisfying $|f^{-1}(a)-\alpha^0| <\delta$, given by the map $a(f,\psi)=f\left(\alpha(f,\psi)\right)$. Furthermore, $a$ is a nondegenerate critical point of $W^\O(\cdot , g)$.
\end{proposition}
\noindent
Here, $V_\beta$ is given by \eqref{e3.23a}.
The  proof of Proposition \ref{p301} is identical to the one of Proposition \ref{propndstable} and is left to the reader.

We will need later the following special case of Proposition~\ref{p301}, where $\O$ is fixed.
\begin{corollary}\label{a(g)}
Let $a^0\in\Omega_*^k$ be a nondegenerate critical point of $W(\cdot,g_0)$, for some $g_0\in C^{1,\beta}(\partial\Omega;\S^1)$. Then, for $g$ in a small $C^{1,\beta}$-neighborhood $\mathfrak{A}$ of $g_0$, $W(\cdot, g)$ has, near $a^0$,  a unique nondegenerate critical point $a(g)$. In addition, the map $\psi\mapsto a(g_0 e^{\im\psi})$, defined for $\psi$ in a sufficiently small neighborhood of the origin in $C^{1,\beta}(\partial\Omega;\R)$, is smooth. 
\end{corollary}
We note here that Corollary~\ref{a(g)} allows us to define a map 
\be
\label{defT*}
T_*=T^\Omega_{*,a^0,g_0}:\mathfrak{A}\to C^\beta(\partial\Omega;\R),\  
T_*(g) := N^\Omega(a(g),g) = u_{*,a(g),g}\wedge\frac{\partial u_{*,a(g),g}}{\partial\nu}.
\end{equation}
Since $W^\O(\cdot,g)$ does not depend on the class of $g$ modulo $\mathbb{S}^1$, neither do $a(g)$ and $T_\ast$. Moreover, in view of \eqref{e21121} and \eqref{e21122} we have
\begin{equation*}
\int_{\partial\Omega} u_{*,a,g} \wedge \frac{\partial u_{*,a,g}}{\partial\nu}  =\int_{\po} \frac{\partial\Phi_{a,g}}{\partial\tau}=0.
\end{equation*}
We find that
the map $T_*$ induces a map, still denoted $T_\ast$, from  $\mathfrak{A}/\mathbb{S}^1$ into $\dot{C}^\beta(\partial\Omega;\R)$. Here, we define
\bes
\dot{C}^\beta(\partial\Omega;\R):=\left\{ \psi\in C^\beta(\po ; \R);\, \int_{\po}\psi=0\right\}.
\ees
It is also convenient to consider, in a sufficiently small neighborhood $\mathfrak{B}$  of the origin  in $C^ {1,\beta}(\partial\Omega;\R)$,  the maps (both denoted $U_\ast$)
\be
\l{e306}
U_*=U^\Omega_{*,a^0,g_0}:\mathfrak{B} \to \dot{C}^\beta(\partial\Omega;\R),\  
U_*(\psi) = T_*(g_0 e^{\im\psi})
\ee
and 
\begin{equation}\label{U*}
U_*=U^\Omega_{*,a^0,g_0}:\mathfrak{B}/\R \to \dot{C}^\beta(\partial\Omega;\R),\  
U_*(\psi) = T_*(g_0 e^{\im\psi}).
\end{equation}
The above $U_*$'s are smooth. Indeed, this is obtained by combining \eqref{transformN} with \eqref{Ndisc} and with the fact that  $\psi\mapsto a(g_0 e^{\im\psi})$ is smooth.

\section{A uniform version of the Pacard-Rivi\`ere construction of critical points}
\l{prus}
We start by explaining how the results established in this section compare to the existent literature. 

Let us first briefly recall the Bethuel-Brezis-H\'elein analysis of critical points  of the Ginzburg-Landau energy $E_\ve$ with prescribed Dirichlet boundary condition $g\in C^\infty(\po ; \so)$ \cite[Chapter X]{bbh}.  Consider a fixed boundary condition $g\in C^{\infty}(\partial\Omega ; \mathbb{S}^1)$, with $\mathrm{deg}(g,\partial\Omega)=d=\sum_{j=1}^k d_j$. Given a critical point $a\in\Omega_*^k$  of $W(\cdot, d_1, \ldots, d_k, g)$, consider the canonical harmonic map given by \eqref{u*1}. The fact that $a$ is a critical point of $W(\cdot, d_1, \ldots, d_k, g)$ is equivalent to the fact that the harmonic function $H_j$, defined near $a_j$ by
\begin{equation}\label{Hj1}
u_*=e^{\im H_j}\left(\frac{z-a_j}{|z-a_j|}\right)^{d_j},
\end{equation}
satisfies $\nabla H_j(a_j)=0$ \cite[Chapter VII]{bbh}.

%
%
The main result in 
 \cite[Chapter X]{bbh} asserts that, when $\O$ is starshaped critical points of $E_\ve$ converge, as $\ve\to 0$ (up to subsequences and in  appropriate function spaces), to a canonical harmonic map $u_*=u_{*, a, g}$ associated with a critical point $a\in\Omega_*^k$  of $W(\cdot, d_1, \ldots, d_k, g)$.
 
%
Granted this result, one can address the converse: given a critical point $a\in\Omega_*^k$ of $W(\cdot, d_1,\ldots, d_k, g)$, does there exist critical points $u_{\varepsilon}$ of $E_{\varepsilon}$ with prescribed boundary condition $g$, such that $u_{\varepsilon}\longrightarrow u_{*,a,g}$ as $\varepsilon\to 0$? Here we will be interested in the answer provided by  Pacard and  Rivi\`ere \cite{pacardriviere}.

\begin{theorem}[{\cite[Theorem~1.4]{pacardriviere}}] \label{pr} Let $0<\beta,\gamma<1$. 
Assume that $g\in C^{2,\beta}(\partial\Omega ; \mathbb{S}^1)$ and $d_j\in\lbrace\pm 1\rbrace$. Let $a\in\Omega_*^k$ be a nondegenerate critical point of $W(\cdot, d_1,\ldots, d_k, g)$.

Then there exists $\varepsilon_0>0$ such that for every $\varepsilon\in(0,\varepsilon_0)$, there exists $u_{\varepsilon}$ a critical point of $E_{\varepsilon}$ with $u_{\varepsilon}=g$ on $\partial\Omega$, and
\begin{equation}\label{cvpr}
u_{\varepsilon}\longrightarrow u_{*, a(g),g}\quad\text{as }\varepsilon\to 0
\end{equation}
in $C^{2,\gamma}_{loc}(\Omega\setminus\lbrace a_1,\ldots,a_k\rbrace)$.
\end{theorem}
The purpose of this section is to establish a variant of Theorem \ref{pr}, in which $g$ is assumed to be merely $C^{1,\beta}$ and is not fixed anymore. In addition, we will obtain a uniform existence theorem, and uniform convergence rate. 
More specifically, we fix integers $d_1,\ldots, d_k$. Since these integers do not depend on the boundary datum $g$ we consider, we will omit the dependence of $W$ with respect to $d_1,\ldots, d_k$: we write $W(\cdot, g)$ instead of $W(\cdot, d_1, \ldots, d_k, g)$. We consider $a_0\in\Omega_*^k$ a nondegenerate critical point of the renormalized energy $W(\cdot, g_0)$ associated with $g_0\in C^{1,\beta}(\partial\Omega ; \mathbb{S}^1)$. By Corollary~\ref{a(g)}, we know that, for $g$ in a small $C^{1,\beta}$-neighborhood $\mathfrak{A}$ of $g_0$, $W(\cdot, g)$ has, near $a_0$,  a unique nondegenerate critical point $a(g)$. 

%
In this section, we establish the following variant of Theorem \ref{pr}.
\begin{theorem}\label{pru}
Let $0<\beta,\gamma<1$. Let $g_0\in C^{1,\beta}(\po ; \so)$. Let $d_1,\ldots, d_k\in \{ -1, 1\}$. Let $a_0$ be a nondegenerate critical point of $W(\cdot, g_0)$. 
Then there exist $\delta>0$ and $\varepsilon_0>0$ such that the following holds. For every $g\in C^{1,\beta}(\partial\Omega ;\mathbb{S}^1)$ satisfying $\|g-g_0\|_{C^{1,\beta}}\leq \delta$, and for every $\varepsilon\in (0,\varepsilon_0)$, there exists $u_{\varepsilon}=u_{\varepsilon,g}$ a critical point of $E_{\varepsilon}$ with prescribed boundary condition $g$, such that
\begin{equation}\label{cvpru}
u_{\varepsilon,g}\longrightarrow u_{*, a(g), g}\quad\text{as }\varepsilon\to 0
\end{equation}
in $C^{2,\gamma}_{loc}(\Omega\setminus\lbrace a_1,\ldots,a_k\rbrace)$.
\end{theorem}
As announced,  the difference with  Theorem 1.4 in  \cite{pacardriviere} is that we merely assume  that $g\in C^{1,\beta}$; in addition,   we prove that $\ve_0$ can be chosen independent of $g$.   
Theorem~\ref{pru} allows us to define a map $F_{\varepsilon} :g\mapsto u_{\varepsilon,g}$ for every $\varepsilon\in (0,\varepsilon_0)$.

Theorem~\ref{pru} is obtained by following the proof  of Theorem~\ref{pr} in \cite{pacardriviere}. All we have to check (and we will do in what follows) is that the estimates in \cite{pacardriviere}  are uniform in $g$; we also have to modify some arguments relying on the regularity assumption  $g\in C^{2,\beta}$.  

\begin{proof}[Proof of Theorem \ref{pru}]
For the convenience of the reader, we recall the main steps of the proof of Theorem~\ref{pr} in \cite{pacardriviere}, and examine the crucial points where the estimates  depend on $g$, respectively  where the regularity of $g$ plays a role.

The general strategy in \cite{pacardriviere} is to construct an \enquote{approximate solution}  $\widetilde{u}_\varepsilon$ of the Ginzburg Landau equation
\begin{equation}\label{Neps}
\mathcal{N}_{\varepsilon}(u)= 0,\quad\text{where }\mathcal{N}_{\varepsilon}(u):=\Delta u + \frac{u}{\varepsilon^2}(1-|u|^2),
\end{equation}
using the fairly precise knowledge we have of the form of solutions for small $\varepsilon$. Then, using a fixed point argument, one can prove that some perturbation of $\widetilde{u}_{\varepsilon}$ is in fact an exact solution of \eqref{Neps}. The main difficulty lies in finding the good functional setting that makes the linearized operator $L_\varepsilon = D\mathcal{N}_\varepsilon$ around $\widetilde{u}_\varepsilon$  invertible, uniformly with respect to $\varepsilon$. This is achieved in \cite{pacardriviere} in the frame of appropriate weighted Hölder spaces.

In \cite{pacardriviere} the proof of Theorem~\ref{pr} is divided into five chapters: Chapters 3 through 7. In what follows, we detail the content of these chapters and explain how to adapt the arguments for the need of Theorem~\ref{pru}.

\subsection*{Chapters 3 and 4 in \cite{pacardriviere}}

 \cite[Chapter~3]{pacardriviere} is devoted to the study of the radially symmetric solution $u(r e^{\im\theta})=f(r)e^{\im\theta}$ of the Ginzburg-Landau equation in  $\C$ satisfying $\lim_{r\to\infty}f(r)=1$. In particular,  \cite[Chapter~3]{pacardriviere} characterizes the bounded solutions  of the linearized equation about this radial solution. This characterization is used in \cite[Chapter~4]{pacardriviere} in the  study of the mapping properties of the linearization of the Ginzburg-Landau operator (at the radial solution) in the punctured unit disc $\mathbb{D}\setminus\lbrace 0\rbrace$. In particular,  it is shown that the linearized operator is invertible between appropriate weighted Hölder spaces.
 
These two chapters (3 and 4) are independent of the boundary condition $g$, so that they can be used with no changes in the proof of Theorem \ref{pru}.

\subsection*{Chapter 5 in \cite{pacardriviere}}

The next step, in \cite[Chapter~5]{pacardriviere}, consists in constructing and estimating the approximate solution $\widetilde{u}_{\varepsilon}$. This approximate solution looks like $u_*=u_{*,g,a(g)}$ away from its zeros (which are close to the singularities of $u_*$), and like the radial solution studied in \cite[Chapter~3]{pacardriviere} near its zeros. 
Since $\widetilde{u}_\varepsilon$ is built upon $u_*$, the estimates satisfied by $\widetilde{u}_\varepsilon$ involve $u_*$, and thus $g$.

More specifically, in \cite[Chapter~5]{pacardriviere}, various quantities are estimated in terms of constants $c(u_*)$ depending on $u_*$ and its derivatives. An inspection of the proofs there combined with \eqref{u*1} shows that these constants depend only on   $a(g)$, on the harmonic function $H=H_g$ and on the  derivatives of $H_g$.

We claim that the constants $c(u_*)$ can be chosen independent of $g$ satisfying 
\be
\l{e11151}\|g-g_0\|_{C^{1,\beta}(\po)}\leq\delta.
\ee 
Here, $\delta$ is sufficiently small in order to have the conclusion of Corollary~\ref{a(g)}.
Indeed, the key observation is that there exists a constant $C>0$ independent of $g$ such that
\begin{equation}\label{estimH}
\|H\|_{C^{1,\beta}(\overline{\Omega})}\leq C;
\end{equation}
this follows from the fact that  $H$ is harmonic and $\|H\|_{C^{1,\beta}(\partial\Omega)}\leq C$. 

In particular, we have
\begin{equation}\label{interiorestimH}
\| H\|_{C^k(\omega)}\leq C(k,\omega)\quad\text{for } k\in\N\text{ and }\overline{\omega}\subset\Omega.
\end{equation}
Estimate \eqref{interiorestimH} implies that all the interior estimates in \cite[Chapter~5]{pacardriviere} are satisfied uniformly in  $g\in C^{1,\beta}$ satisfying \eqref{e11151}. This settles the case of estimates (5.8), (5.9), (5.33), (5.42) and (5.43) in \cite[Chapter~5]{pacardriviere}.

It remains to consider the global and boundary estimates (5.29), (5.32) and (5.41) in \cite{pacardriviere}. These estimates rely on bounds on the solution  $\xi$ of the problem
\begin{equation}\label{eqxi}
\left\lbrace
\begin{array}{rll}
\Delta\xi-|\nabla u_*|^2\xi+\d\frac{1-\xi^2}{\varepsilon^2}\xi & = 0 &\text{in }\Omega_{\delta/2}\\
\xi & = S_{\varepsilon}+w_{j,r} & \text{on }\partial \mathbb{D}_{\delta/2}(a_j)\\
\xi & = 1 & \text{on }\partial\Omega
\end{array}.
\right.
\end{equation}
Here,  $\Omega_\sigma:=\Omega\setminus\bigcup_j\mathbb{D}_\sigma (a_j)$ (for sufficiently small $\sigma>0$), and $\delta:=\varepsilon^2$. The auxiliary function  $S_{\varepsilon}$ is independent of $g$ and is defined  in  \cite[Section 3.6]{pacardriviere}.  Finally,  $w_{j,r}$, defined in \cite[(5.7)]{pacardriviere}, depends only $a(g)$ and on  the restriction of $H_j$ to compacts of $\O$; therefore, the estimates involving $w_{j,r}$ are uniform in $g$. 
%

In \cite[Lemma~5.1]{pacardriviere} the following estimates (numbered as (5.29) in \cite{pacardriviere}) are shown to hold:
\begin{align}
\label{estimxi1}
1-c\varepsilon^2  \leq\xi \leq 1 &\quad\quad \text{in }\Omega_\sigma,\\
\label{estimxi2}
1-c\varepsilon^2r_j^{-2}  \leq \xi  \leq 1 
&\quad\quad \text{in }\mathbb{D}_\sigma(a_j)\setminus\mathbb{D}_{\delta/2}(a_j),\\
\label{estimxi3}
|\nabla^k\xi|  \leq c_k\varepsilon^2 r_j^{-2-k} & \quad\quad\text{in }\mathbb{D}_{2\sigma}(a_j)\setminus\mathbb{D}_\delta(a_j)\quad (k\geq 1).
\end{align}
Here $\sigma>0$ is fixed, and $r_j=r_j(x)$ denotes the distance from $x$ to $a_j$. Estimates \eqref{estimxi2} and \eqref{estimxi3} are interior estimates, and therefore they  hold uniformly in  $g\in C^{1,\beta}$ satisfying \eqref{e11151}, as explained above. 
We claim that the same conclusion applies to \eqref{estimxi1}. Indeed, an inspection of the proof in \cite{pacardriviere} shows that the constant $c$ in  \eqref{estimxi1} is controlled by  $\sup_{\Omega_\sigma}|\nabla u_*|$. The latter quantity is uniformly bounded, thanks to \eqref{estimH}, whence the conclusion.
%
%
This settles the case of the estimate (5.29) in \cite{pacardriviere}.

We next turn to the estimate (5.32) in \cite[Lemma~5.2]{pacardriviere}. Under the assumption that $g\in C^{2,\beta}$, this lemma asserts that
\begin{equation}\label{estimxi4}
\sup_{\Omega_\sigma} |\nabla^k\xi| \leq c \varepsilon^{2-k}, \quad k=1,2.
\end{equation}
In our case, we only assume $g\in C^{1,\beta}$. The corresponding estimates are given by our next result.
\begin{lemma}\label{estimxi}
Assume that \eqref{e11151} holds. Then we have
\begin{equation}\label{estimxi5}
\sup_{\Omega_\sigma}|\nabla\xi|\leq c\varepsilon \text{ and }
|\nabla\xi|_{\beta,\Omega_\sigma}\leq c\varepsilon^{1-\beta}.
\end{equation}
\end{lemma}
Here,  $|\cdot|_{\beta,\Omega_\sigma}$ denotes the $C^{\alpha}$ semi-norm in $\Omega_\sigma$:
\bes
|u|_{\alpha,\Omega_\sigma}:=\sup\left\{ \frac{|u(x)-u(y)|}{|x-y|^\beta};\, x, y\in \Omega_\sigma\right\}.
\ees
\begin{proof}
We apply Lemma~\ref{estimlemma} in the Appendix with $w=\xi-1$ in $G:=\Omega_{\sigma/2}$, and find that
\begin{gather}
\l{fag1}
\sup_{\Omega_\sigma}|\nabla\xi| \leq C\left( \|w\|_{L^\infty(\Omega_\sigma)}^{1/2}\|\Delta w\|_{L^\infty(\Omega_\sigma)}^{1/2}+\| w\|_{C^{1,\beta}(\partial\Omega_\sigma)}\right),\\
\l{fag2}
|\nabla\xi|_{\beta,\Omega_\sigma} \leq C \left(\|w\|_{L^\infty(\Omega_\sigma)}^{1/2-\beta/2}\|\Delta w\|_{L^\infty(\Omega_\sigma)}^{1/2+\beta/2} +\| w\|_{C^{1,\beta}(\partial\Omega_\sigma)}\right).
\end{gather}
The conclusion then follows by combining \eqref{fag1}-\eqref{fag2} with the equation \eqref{eqxi} and with with estimates \eqref{estimxi1} and \eqref{estimxi3}.
\end{proof}
Finally, we examine  estimate (5.41) in the last section of \cite[Chapter~5]{pacardriviere}; this   is a global estimate for $\mathcal{N}_{\varepsilon}(\widetilde{u}_\varepsilon)$. Recall here that $\mathcal{N}_{\varepsilon}$ is the Ginzburg-Landau operator, and that  $\widetilde{u}_\varepsilon$ is the approximate solution of  \eqref{Neps} constructed in \cite[Chapter~5]{pacardriviere}. The estimate \cite[(5.41)]{pacardriviere} involves an  interior estimate and a boundary estimate. As above, the interior estimate is  settled with the help of \eqref{interiorestimH}. We now turn to the boundary estimate, which is the following: 
\begin{equation}\label{estimN}
\|\mathcal{N}_\varepsilon(\widetilde{u}_\varepsilon)\|_{C^{\beta}(\Omega_\sigma)}\leq c\varepsilon^{1-\beta}.
\end{equation}
The proof of \eqref{estimN} in \cite{pacardriviere} relies on the estimates \eqref{estimxi4} above (see\cite[Proof of Lemma~5.2]{pacardriviere}). In our case, \eqref{estimxi4} need not hold, since we only assume that $g\in C^{1,\beta}$. However, we still obtain \eqref{estimN} as follows. 
We note that
\begin{equation}\label{formuleN}
\mathcal{N}_\varepsilon(\widetilde{u}_\varepsilon) = 2\nabla u_*\cdot\nabla\xi \quad\text{in }\Omega_\sigma
\end{equation}
(this is formula (5.46) in \cite{pacardriviere}). By \eqref{formuleN}, we have
\begin{equation}\label{estimN2}
\|\mathcal{N}_\varepsilon(\widetilde{u}_\varepsilon)\|_{C^{\beta}(\Omega_\sigma)}\leq
c\|\nabla u_*\|_{C^{\beta}(\Omega_\sigma)} \|\nabla\xi\|_{C^{\beta}(\Omega_\sigma)}.
\end{equation}
We obtain \eqref{estimN} as a consequence of \eqref{estimH} and of Lemma~\ref{estimxi}.

As a conclusion of this inspection, we find that  all the estimates in  \cite[Chapter~5]{pacardriviere} are uniform in $g$ satisfying \eqref{e11151}; the arguments there need only minor changes. 
The most relevant change is that  \cite[Lemma~5.2]{pacardriviere} has to be replaced by Lemma~\ref{estimxi}.

\subsection*{Chapter 6 in \cite{pacardriviere}}

We now turn to \cite[Chapter~6]{pacardriviere}, which deals with the conjugate linearized operator $\widetilde{\mathcal{L}}_\varepsilon$ around the approximate solution. The main result of this chapter is \cite[Theorem~6.1]{pacardriviere}, which  states that $\widetilde{\mathcal{L}}_\varepsilon$ is invertible for $\varepsilon\in (0,\varepsilon_0)$, with the norm of its inverse bounded independently of $\varepsilon$. In order to adapt this theorem to our situation, we need to check that this $\varepsilon_0$, and the bound on $\widetilde{\mathcal{L}}_\varepsilon^{-1}$, can be chosen independently of $g$ satisfying \eqref{e11151}.

The proof of \cite[Theorem~6.1]{pacardriviere} is divided into three parts:
\begin{enumerate}[(a)]
\item The \enquote{interior} problem, consisting in the study of the linearized operator $\widetilde{\mathcal{L}}_\varepsilon$ near the zeros of $\widetilde u_\ve$ \cite[Section~6.2]{pacardriviere}.
\item The \enquote{exterior} problem, requiring the study of the linearized operator $\widetilde{\mathcal{L}}_\varepsilon$ away from the zeros of $\widetilde u_\ve$ \cite[Section~6.3]{pacardriviere}.
\item The  study of the Dirichlet-to-Neumann mappings \cite[Section~6.4]{pacardriviere}. (These mappings are used later in order to  \enquote{glue}  the two first steps together.)
\end{enumerate}

The interior and the exterior problem rely on the estimates obtained in \cite[Chapter~5]{pacardriviere}. An inspection of the proofs shows that all the estimates obtained there are uniform in $g$, with one possible exception: the estimates in \cite[Proposition~6.2]{pacardriviere}. Indeed, these estimates rely on \cite[Lemma~5.2]{pacardriviere}, and more specifically on \eqref{estimxi4} (which does not hold in our setting). However, a closer look to  \cite[Proof of Proposition 6.2]{pacardriviere} shows that the conclusion of \cite[Proposition 6.2]{pacardriviere} still holds if we replace \eqref{estimxi4} by Lemma~\ref{estimxi}. 
In conclusion, the first two steps can be carried out with uniform estimates, provided \eqref{e11151} holds.

The third step (Dirichlet-to-Neumann mappings) requires more care. In \cite[Section~6.4]{pacardriviere}, the following  two operators are defined, for fixed small $\zeta>0$ and for sufficiently  small $\varepsilon$:
\begin{equation}\label{DtN}
DN_{int,\varepsilon},\: DN_{ext,\varepsilon} : \prod_{j=1}^k C^{2,\beta}\left(C(\zeta, a_j)\right)\longrightarrow \prod_{j=1}^k C^{1,\beta}\left(C(\zeta, a_j)\right).
\end{equation}
(These are the  interior and exterior Dirichlet-to-Neumann mappings.) The crucial result in part (c) is \cite[Proposition 6.5]{pacardriviere}, which states the existence of some  $\varepsilon_0$ such that $DN_{int,\varepsilon}-DN_{ext,\varepsilon}$ is an isomorphism for $\varepsilon\in (0,\varepsilon_0)$. The  proof of this fact goes as follows. First the convergence
\begin{equation}\label{cvDtN}
DN_{int,\varepsilon}-DN_{ext,\varepsilon}\longrightarrow DN_{int,0}-DN_{ext,0}\quad\text{as }\varepsilon\to 0
\end{equation}
is shown to hold in operator norm. The proof of  \eqref{cvDtN} relies on the interior estimate \eqref{interiorestimH}. Therefore, the convergence in \eqref{cvDtN} is uniform in $g$ satisfying \eqref{e11151}. 

We now return to the proof in \cite[Chapter 6]{pacardriviere}. Once \eqref{cvDtN} is established, it remains to prove that the limiting operator $DN_{int,0}-DN_{ext,0}$ is invertible. This is done in \cite[Proposition 6.5]{pacardriviere}; this is where the nondegeneracy of $a$ as a critical point of $W(\cdot, g)$ comes into the picture. In order to extend the conclusion of \cite[Proposition 6.5]{pacardriviere} to our setting, and to obtain a uniform bound for the inverse of $DN_{int,\varepsilon}-DN_{ext,\varepsilon}$, it suffices to check that $DN_{int,0}-DN_{ext,0}$ depends continuously on $g$. Indeed, this will lead to a uniform bound for the inverse of $DN_{int,\varepsilon}-DN_{ext,\varepsilon}$ provided $\ve$ is sufficiently small, uniformly in $g$ satisfying \eqref{e11151} (possibly with a smaller $\delta$).
%
%
For this purpose, we examine the formulas of $DN_{ext,0}$ and $DN_{int,0}$. 
The definition of $DN_{ext,0}$ is given in \cite[Proposition~6.4]{pacardriviere}, and it turns out that  that $DN_{ext,0}$ does not depend on $g$. As for $DN_{int,0}$, it is a diagonal operator of the form
\begin{equation}\label{DNint}
DN_{int,0}(\phi_1,\ldots,\phi_k)=\left(DN^1_{int,0}(\phi_1),\ldots , DN^k_{int,0}(\phi_k)\right),
\end{equation}
with $DN^j_{int,0} : C^{2,\beta}\left(C(\zeta, a_j)\right)
\to C^{1,\beta}\left(C(\zeta, a_j)\right)$, $\fo j\in \llbracket 1,k\rrbracket$.

Furthermore, from \cite[Proposition~6.3]{pacardriviere} we know that $DN^j_{int,0}$ further splits as
\begin{equation}\label{DNintsepvar}
DN^j_{int,0} = T_1 \oplus T_2,\quad\text{with }
\begin{cases}
T_1:\mathrm{span}\left\lbrace e^{\pm \im n\theta}\right\rbrace_{n\geq 2} \to \mathrm{span}\left\lbrace e^{\pm \im n\theta}\right\rbrace_{n\geq 2} \\
T_2:\mathrm{span}\left\lbrace e^{\pm \im\theta}\right\rbrace\to\mathrm{span}\left\lbrace e^{\pm \im\theta}\right\rbrace
\end{cases}.
\end{equation}
Here, the operator $T_1$ does not depend on $g$. Therefore,  we only need to check that $T_2$ depends continuously on $g$. As a linear operator on a two-dimensional space, $T_2$ is represented by a $2\times 2$ matrix. It is clear from \cite[Proposition~6.3]{pacardriviere} that the coefficients of this matrix are smooth functions of $\nabla^2 H(a_j)$. In turn, $\nabla^2 H(a_j)$ depends smoothly on $g$, by Corollary~\ref{a(g)}.

Hence $DN_{int,0}-DN_{ext,0}$ depends continuously on $g$, as claimed.

This allows us to choose $\varepsilon_0$ independent of $g$ satisfying \eqref{e11151} in \cite[Proposition~6.5]{pacardriviere} and in \cite[Theorem~6.1]{pacardriviere}, and to obtain a uniform estimate for the inverse of $\widetilde{\mathcal{L}}_\varepsilon$.


\subsection*{Chapter 7 in \cite{pacardriviere}}

Finally, in \cite[Chapter~7]{pacardriviere} the results and estimates in \cite[Chapters 3-6]{pacardriviere}  are combined in order to prove Theorem \ref{pr}. Our above analysis shows that these estimates are uniform, and therefore lead to the uniform version Theorem \ref{pru} of  Theorem \ref{pr}.

\subsection*{Conclusion}
As a conclusion of our analysis, Theorem \ref{pru} holds.
\end{proof}
For further use, we record two additional properties of the maps $u_{\ve, g}$; these properties are immediate consequences of the construction in \cite{pacardriviere}. Let $\delta$ be as in Theorem \ref{pru}. We consider the set
\be
\l{e301}
{\mathfrak A}:=\left\{g\in C^{1,\beta}(\po ; \so);\, \|g-g_0\|_{C^{1,\beta}}<\delta\right\}.
\ee
\bl
\l{l301} 
Let $K\Subset \overline{\Omega}\setminus \lbrace a_j \rbrace$. Then
we have $|u_{\varepsilon,g}|\to 1$ as $\varepsilon\to 0$, uniformly in $K$ and in $g\in\mathfrak A$.
\el
\bp
This follows by an inspection of the construction in \cite{pacardriviere}. Formulas (5.36) and (5.37) in \cite{pacardriviere} ensure that,  for small $\varepsilon$, the approximate solution $\widetilde u_\ve$ satisfies  $|\widetilde{u}_\varepsilon|=|\xi|$ in $K$. The convergence then follows from the estimates on $\xi$, and  from formula (7.1) in \cite{pacardriviere} connecting the approximate solution to the exact solution.
\ep
For the next result, it may be necessary to replace $\delta$ by a smaller value.
\bl
\l{l302}
Let $g\in\mathfrak A$ and $\omega\in\so$. If $\omega g\in \mathfrak A$, then $u_{\ve, \omega g}=\omega u_{\ve, g}$. 
\el
\bp
We have $W(\cdot, g)=W(\cdot, \omega g)$. Therefore, if $a$ is a nondegenerate critical point of $W(\cdot, g)$, then $a$ is also a nondegenerate critical point of $W(\cdot, \omega g)$. By Corollary \ref{a(g)},  we find that $a(\omega g)=a(g)$. Using this equality, an inspection of the construction in \cite{pacardriviere} shows that  
\be
\l{e302}
\widetilde u_{\ve,\omega g}=\omega \widetilde u_{\ve, g}.
\ee 
Thanks to \eqref{e302}, we obtain that $\omega u_{\ve, g}$ has all the properties satisfied by the solution $u_{\ve,\omega g}$ constructed from $\widetilde u_{\ve,\omega g}$ via the inverse function theorem. Since the solution provided by the inverse function theorem is unique, we find that $u_{\ve, \omega g}=\omega u_{\ve, g}$, as claimed.
\ep

\section{Convergence of the normal differentiation operators}\label{sconv}
In this section, we fix integers $d_1,\ldots, d_k\in \{ -1, 1\}$ as in Section \ref{prus}. We assume that $a^0$ is a nondegenerate critical point of $W(\cdot, g_0)$. Let $g\in\mathfrak A$, where $\mathfrak A$ is given by \eqref{e301}, and let $0<\ve<\ve_0$. For such $g$ and $\ve$, we define   $u_\ve=u_{\ve, g}$ as in Section \ref{prus}. We also define $u_{\ast, g}:=u_{\ast, a(g), g}$, where $a(g)$ is the unique critical point of $W(\cdot, g)$ close to $a_0$ (see Corollary~\ref{a(g)}). We consider   the operators
\bes
T_\ve, T_\ast:{\mathfrak A}\to C^\beta(\po ;\R),\ T_\ve(g):=u_{\ve, g}\wedge\frac{\p u_{\ve, g}}{\p\nu}\text{ and }T_\ast(g):=u_{\ast, g}\wedge\frac{\p u_{\ast, g}}{\p\nu}.
\ees
The main result of this section is the following
\bpr
\l{p12031}
Let $0<\gamma<1$. Then (possibly after replacing $\delta$ by a smaller number) we have
\be
\l{e12031}
\lim_{\ve \to 0}\sup_{g\in{\mathfrak A}}\left\|  T_\ve(g)-T_\ast(g)\right\|_{C^\gamma(\po)}=0.
\ee
In particular, given $\mu>0$ there exists some $\ve_0>0$ such that, for $0<\ve<\ve_0$, $T_\ve-T_\ast:{\mathfrak A}\to C^\beta(\po ; \R)$ is compact and satisfies 
\bes
\|T_\ve(g)-T_\ast(g)\|_{C^\beta(\po)}\le\mu,\quad \fo \ve<\ve_0,\ \fo g\in{\mathfrak A}.
\ees
\epr
\begin{proof}
The last part of the proposition follows from the fact that the embedding $C^\gamma(\po ; \R)\hookrightarrow C^\beta(\po ; \R)$ is compact when $\gamma>\beta$. 

Whenever needed in the proof, we will replace $\delta$ by a smaller number. Let $a=a(g)$, $g\in\mathfrak A$, be such that $\na_a W(a, g)=0$ and $a$ is close to $a^0=(a^0_1,\ldots a^0_k)$. Let $t>0$ be a small number and set 
\bes
\omega:=\left\{ x\in\O;\, |x-a^0_j|>t,\ \fo j\in \llbracket 1,k\rrbracket\right\}.
\ees
We may assume that $|a(g)-a^0|<t/2$, $\fo g\in\mathfrak A$. In view of Theorem \ref{pru}, we have $u_{\ve, g}\to u_{\ast, g}$ in $C^{2,\gamma}(K)$ as $\ve\to 0$, for every $g\in\mathfrak A$ and for every $K$ compact set such that $K\subset \overline\omega\setminus \po$. In addition, by Lemma \ref{l301} we have $|u_{\ve, g}|\to 1$ as $\ve\to 0$ uniformly in $\overline\omega$ and in $g\in\mathfrak A$. 

Let $\theta=\theta_g$ be the multi-valued argument of 
\bes
z\mapsto \prod_{j=1}^k\frac{\left(z-a_j(g)\right)^{d_j}}{\left|z-a_j(g)\right|^{d_j}}. 
\ees
We note that $\na\theta_g$ is single-valued and that we have
\be
\l{e1203a}
\|\na \theta_g\|_{C^{1,\beta}(\omega)}\le C,\quad\fo g\in{\mathfrak A}.
\ee
For small $\ve$ (independent of $g$), we have $\deg (u_{\ve, g})=\deg (u_{\ast, g})=d_j$ on $C(a^0_j,t)$, and thus we may write, locally in $\overline\omega$, 
\bes
u_{\ve, g}=\rho e^{\im\va}=\rho_{\ve, g} e^{\im\va_{\ve, g}}=\rho e^{\im(\theta+\psi)}=\rho_{\ve, g} e^{\im(\theta_g+\psi_{\ve, g})},
\ees
and similarly
\bes
u_{\ast, g}= e^{\im(\theta+\psi_\ast)}= e^{\im(\theta_g+\psi_{\ast, g})}.
\ees
We may choose $\psi_{\ast ,g}$ in order to have
\be
\l{e1203b}
\| \psi_{\ast, g}\|_{C^{1,\beta}(\omega)}\le C,\quad\fo g\in{\mathfrak A},
\ee
and we normalize $\psi_{\ve, g}$ by the condition
\be
\l{e1203c}
\psi_{\ve, g}=\psi_{\ast, g}\quad\text{on }\po.
\ee
In terms of $\rho$, $\va$ and $\psi$, the  Ginzburg-Landau equation reads
\bes
\begin{cases}
\div (\rho^2\na\va)=\div (\rho^2(\theta+\psi))=0\\
\d -\Delta \rho=\frac 1{\ve^2}\rho(1-\rho^2)-\rho|\na\va|^2
\end{cases}.
\ees
\smallskip
\noindent
{\it Step 1.} We have 
\bes
\|\na\va_{\ve, g}\|_{L^p(\omega)}\le C_p,\quad, \fo \ve<\ve_0\ \fo g\in{\mathfrak A},\ \fo 1<p<\infty.
\ees
Indeed, we start by noting that we have $\|\na\theta_{ g}\|_{L^p(\omega)}\le C_p$; therefore, it suffices to prove that $\|\na\psi_{\ve, g}\|_{L^p(\omega)}\le C_p$. Using the equation $\div (\rho^2\na\va)=0$, we see  that $\psi_{\ve, g}$ satisfies
\be
\l{e12031b}
\Delta\psi_{\ve, g}=\div \left(\left(1-\rho^2_{\ve,g}\right)\na\theta_g+\left(1-\rho^2_{\ve,g}\right)\na\psi_{\ve, g}\right)\quad\text{in }\omega.
\ee
We obtain
\bes
\begin{split}
\|\na\psi_{\ve, g}\|_{L^p(\omega)} & \le C\left( \|\psi_{\ve, g}\|_{W^{1-1/p,p}(\p\omega)}+\|(1-\rho^2_{\ve,g})\na\theta_g\|_{L^p(\omega)}+\|(1-\rho^2_{\ve,g})\na\psi_{\ve, g}\|_{L^p(\omega)} \right) \\
& \le C_p +C\|1-\rho^2_{\ve,g}\|_{L^\infty(\omega)}\|\na\psi_{\ve, g}\|_{L^p(\omega)}.
\end{split}
\ees
Since $\rho_{\ve,g}\to 1$ as $\varepsilon\to 0$ uniformly in $\overline{\omega}$ and in $g\in\mathfrak{A}$, the second term in the right-hand side of the above inequality can be absorbed in the left-hand side and  we obtain the announced result.

\smallskip
\noindent
{\it Step 2.} For $1<p<\infty$ we have $\na\rho_{\ve , g}\to 0$ in $L^p(\omega)$ as $\ve\to 0$, uniformly in $g\in{\mathfrak A}$.\\
This is obtained as follows. Let $\eta:=\eta_{\ve, g}:=1-\rho_{\ve, g}\in [0,1]$, which satisfies
\be
\l{e12032}
\begin{cases}
\d-\Delta \eta+\frac 1{\ve^2}\rho(1+\rho)\eta=\rho|\na\va|^2&\text{in }\omega\\
\eta=0&\text{on }\po.\\
\end{cases}
\ee
Moreover, we have 
\begin{equation}\label{e12032a}
\frac{1}{4\varepsilon^2}\eta\leq \frac{1}{\varepsilon^2}\rho(1+\rho)\eta =\rho|\nabla\varphi|^2+\Delta\eta\leq C\quad\text{on }\partial\omega\setminus\partial\Omega,
\end{equation}
since $u_{\varepsilon,g}\to u_{*,g}$ in $C^{2,\gamma}(K)$ for any compact $K\subset \overline{\omega}\setminus\partial\Omega$, uniformly in $g\in\mathfrak{A}$.

We may assume that $p\ge 2$. Multiplying \eqref{e12032} by $\eta^{p-1}$ and using Step 1, H\"older's inequality and \eqref{e12032a} we find that, for small $\ve$, we have
\bes
\begin{aligned}
\frac 1{4\ve^2}\int_\omega\eta^p & \le \frac 1{\ve^2}
\int_\omega \rho(1+\rho)\eta^p\\
&=\int_\omega \rho|\na\va|^2\eta^{p-1}+\int_{\p\omega\setminus\po}\eta^{p-1}\frac{\p\eta}{\p\nu}-(p-1)\int_\omega\eta^{p-2}|\na\eta|^2\\&
\le \int_\omega \rho|\na\va|^2\eta^{p-1}+\int_{\p\omega\setminus\po}\eta^{p-1}\frac{\p\eta}{\p\nu}
\le C\left(\int_\omega \eta^p\right)^{(p-1)/p}+C\varepsilon^{2(p-1)},
\end{aligned}
\ees
and thus
\be
\l{e120333}
\|\eta_{\ve, g}\|_{L^p(\omega)}\le C_p\varepsilon^2,\quad\fo \ve<\ve_0,\ \fo g\in {\mathfrak A},\ \fo p<\infty.
\ee
Inserting \eqref{e120333} into \eqref{e12032}, we find that $\Delta\eta$ is bounded in $L^p(\omega)$, $\fo p<\infty$. By standard elliptic estimates, we find that $\eta$ (and thus $\rho$) is bounded in $W^{2,p}(\omega)$, $\fo p<\infty$. We conclude via the compact embedding $W^{2,p}\hookrightarrow W^{1,p}$ and the fact that, by Lemma \ref{l301}, we have $\rho\to 1$ uniformly in $\overline\omega$.

\medskip
\noindent
{\it Step 3.} For every $\gamma<1$, we have $\psi_{\ve, g}\to \psi_{\ast, g}$ in $C^{1,\gamma}(\overline\omega)$ as $\ve\to 0$, uniformly in $g\in{\mathfrak A}$.\\
Indeed, $\psi_{\ve , g}-\psi_{\ast, g}$ satisfies
\be
\l{e12034}
\begin{cases}
\Delta (\psi_{\ve, g}-\psi_{\ast, g})=\d -\frac 2{\rho_{\ve,g}}\na\rho_{\ve,g}\cdot \na(\theta_g+\psi_{\ve, g})&\text{in }\omega\\
\psi_{\ve,g}-\psi_{\ast, g}=0&\text{on }\po\\
\psi_{\ve,g}-\psi_{\ast, g}\to 0\text{ in }C^2&\text{on }\p\omega\setminus\po
\end{cases},
\ee
the latter convergence being uniform in $g$. By Steps 1 and 2, we have
\bes
\|\Delta (\psi_{\ve,g}-\psi_{\ast, g})\|_{L^p(\omega)}\to 0\quad\text{ as }\ve\to 0,\text{ uniformly in }g.
\ees
Using \eqref{e12034}, we find that $\psi_{\ve,g}-\psi_{\ast, g}\to 0$ in $W^{2,p}(\omega)$. We conclude via the  embedding $W^{2,p}(\omega)\hookrightarrow C^{1,\gamma}(\overline\omega)$, valid when $p>2$ and $\gamma=1-2/p$.

\medskip
\noindent
{\it Step 4.} Conclusion.\\
We have 
\bes
T_\ve(g)=u_{\ve , g}\wedge\frac{\p u_{\ve, g}}{\p\nu}=\frac{\p\va_{\ve,g}}{\p\nu}=\frac{\p\theta_{g}}{\p\nu}+\frac{\p\psi_{\ve,g}}{\p\nu},
\ees
and similarly $\d T_\ast(g)=\d \frac{\p\theta_{g}}{\p\nu}+\frac{\p\psi_{\ast,g}}{\p\nu}$. Using Step 3, we find that
\bes
T_\ve(g)-T_\ast(g)=\frac{\p(\psi_{\ve,g}-\psi_{\ast, g})}{\p\nu}\to 0\text{ in } C^\gamma(\po)\quad\text{as }\ve\to 0,\ \text{uniformly in }g\in{\mathfrak A}.\qedhere
\ees
\end{proof}

\section{Existence of critical points in nondegenerate domains}
\label{scp}

Before stating the main result of this section, let us recall the definition \eqref{U*} of the operator $U_*$ in Section \ref{snd}. Given $a^0$ a nondegenerate critical point of $W(\cdot,g)$, we first  define,  in a $C^{1,\beta}$ neighborhood of $g$,  the  operator $T_*=T_{*,a^0,g}$. Then $U_*$ is defined in a neighborhood $\mathfrak{B}$ of the origin in $C^{1,\beta}(\partial\Omega;\R)$ by
\bes
\d
U_*(\psi)=U_{\ast, a^0, g}(\psi)= T_*(g e^{\im\psi})=T_{*,a^0,g}(g e^{\im\psi}).
\ees
We still denote by $U_*$ the induced map $U_* : \mathfrak{B}/\R \to \dot{C}^\beta(\partial\Omega;\R)$, and recall that $U_*$ is smooth.

\begin{theorem}
\l{t8.1} Let $d_1,\ldots, d_k\in\{ -1, 1\}$ and set $d:=d_1+\ldots+d_k$. 

Let $\Omega$ be a bounded simply connected $C^{1,\beta}$ domain satisfying the two following nondegeneracy conditions:
\begin{itemize}
\item[(ND1)] There exists $a^0\in\Omega_*^k$ such that $a^0$ is a nondegenerate critical point of $W(\cdot, g^0)=W^\O(\cdot, d_1,\ldots,d_k, g^0)$, with $g^0=g^{a^0}$ the canonical boundary data associated with $a^0$ and $d_1,\ldots, d_k$.
\item[(ND2)] The corresponding operator $U_{*,a^0,g^0} : \mathfrak{B}/\R\to \dot{C}^\beta(\partial\Omega;\R)$ is a local diffeomorphism at the origin, i.e., the differential 
\begin{equation*}
DU_*(0) : C^{1,\beta}(\partial\Omega;\R)/\R \longrightarrow \dot{C}^\beta(\partial\Omega;\R)
\end{equation*}
is invertible.
\end{itemize}
Then there exists $\ve_0>0$ such that, for $\ve\in (0,\ve_0)$, there exists $u_\ve\in\mathcal{E}_d$ a critical point of $E_\ve$ with prescribed degree $d$.
\end{theorem}
\begin{remark}
It will be clear from the proof of Theorem \ref{t8.1} that the nondegeneracy condition {\it (ND2)} can actually be replaced by the following weaker condition:\\
{\it (ND2') $U_*$ is a local homeomorphism near the origin.}

 However in what follows it will be more convenient for us to consider the condition {\it (ND2)}. The main reason for this is that {\it (ND2)} is stable under small perturbation of the domain, while it is not clear that {\it (ND2')} is stable.
\end{remark}
\br
\l{gb1}
We connect here  the hypothesis {\it (ND2)} in Theorem \ref{t8.1} to the hypothesis {\it (ND2)} presented in the introduction. 
As we will see in Section \ref{gen},\footnote{In the special case where $d=1$ and $k=1$, but the arguments there adapt to the general case.} $DU_\ast(0)$ is a Fredholm operator of index zero. Thus   the above hypothesis {\it (ND2)} is equivalent to the fact that $DU_\ast(0)$ is onto. It is not difficult to see (but will not be needed in what follows) that the surjectivity of $DU_\ast(0)$ is equivalent to the hypothesis {\it (ND2)}  in the introduction, and that the index of the operator $L$ that appears in the introduction is 
$\d 
\ind L=\ind DU_\ast(0)+1=1$. 
\er
\begin{proof}[Proof of Theorem \ref{t8.1}]
Since $\Omega$ satisfies {\it (ND1)}, the results of Section~\ref{prus} and \ref{sconv} apply. We consider, as in Section~\ref{sconv}, the operators
\begin{equation*}
T_\ve,T_\ast : \mathfrak{A} \to \dot{C^\beta}(\partial\Omega;\R)
\end{equation*}
and
\bes
U_\ast : \mathfrak{B}/\R \to \dot{C}^\beta(\partial\Omega;\R),
\ees
where $\mathfrak{A} = \lbrace g \in C^{1,\beta}(\Omega,\mathbb{S}^1) ; \; \|g-g^0\|<\delta\rbrace$ and $\mathfrak{B}=\lbrace \psi\in C^{1,\beta}(\partial\Omega;\R);\,\|\psi\|<\delta\rbrace$.
Here, $\delta$ and  $\ve$ are sufficiently small. We note that $T_\ve$ takes its values in $\dot{C}^\beta(\partial\Omega;\R)$. Indeed,  $u=u_{\ve,g}$ satisfies
\begin{equation*}
\int_{\partial\Omega} u \wedge \frac{\partial u}{\partial \nu} = \int_\Omega \div\, (u\wedge \nabla u) = \int_\Omega u \wedge \Delta u = \int_\Omega \frac{|u|^2-1}{\varepsilon^2} u \wedge u =0.
\end{equation*}

By Lemma \ref{l302}, we may also consider the induced operators
\begin{equation*}
U_\ve : \mathfrak{B}/\R \to \dot{C}^\beta(\partial\Omega;\R),\; U_\ve (\psi) = T_\ve(g^0e^{\im\psi}).
\end{equation*}
%
%
By definition of the canonical boundary datum, it holds
\begin{equation}\label{U*(0)}
U_*(0)=u_{*,a^0,g^0}\wedge \frac{\partial u_{*,a^0,g^0}}{\partial\nu} = \frac{\partial \Phi_{a^0,g^0}}{\partial\tau} = \frac{\partial \widehat{\Phi}_{a^0}}{\partial\tau} = 0.
\end{equation}
Thanks to {\it (ND2)}, by considering a smaller $\delta$ if necessary, we may assume that $U_*$ is a homeomorphism onto its image. By \eqref{U*(0)}, there exists some $\eta>0$ such that 
\begin{equation}\label{imU*}
U_*(\mathfrak{B}/\R)\supset \left\lbrace \psi\in \dot{C}^\beta(\partial\Omega;\R);\, \|\psi\|_{C^\beta(\po)}<\eta \right\rbrace:=B_\eta.
\end{equation}
Recall the result of Proposition~\ref{p12031}: for sufficiently small  $\ve$, $U_\ve-U_*$ is compact and we have 
\be
\l{e303}
\lim_{\ve\to 0}\sup\left\{\|U_\ve(\psi)-U_\ast(\psi)\|_{C^\beta(\po)};\,  \psi\in {\mathfrak B}\right\}=0.
\ee
Using \eqref{imU*}, \eqref{e303} and standard properties of the Leray-Schauder degree, we find that 
\begin{equation}\label{imUe}
U_\ve (\mathfrak{B}/\R)\supset \left\lbrace \psi\in \dot{C}^\beta(\partial\Omega;\R);\, \|\psi\|_{C^\beta(\po)}<\frac{\eta}{2} \right\rbrace=B_{\eta/2},
\end{equation}
for sufficiently small $\ve$. Indeed, the argument goes as follows. We start from
%
\begin{equation}\label{imUe2}
U_\ve(\mathfrak{B}/\R)= \left(\text{Id}+(U_\ve-U_*)\circ U_*^{-1} \right) (U_*(\mathfrak{B}/\R)) \supset \left(\text{Id}+(U_\ve-U_*)\circ U_*^{-1} \right)(B_\eta).
\end{equation}
Here, $\text{Id}$ denotes the identity map in $\dot{C}^\beta(\partial\Omega;\R)$. 

Let $L_\ve :=(U_\ve-U_*)\circ U_*^{-1}$. Then  $L_\ve:B_\eta\to\dot{C}^\beta(\partial\Omega;\R)$ is compact and, by \eqref{e303},  there exists $\ve_0>0$ such that
\begin{equation}
\l{e305}
 \sup\left\{ \| L_\ve (\psi)\|_{C^\beta(\partial\Omega)};\, \psi\in B_\eta\right\}< \eta/2,\quad\forall\, \ve\in(0,\ve_0).
\end{equation}
We complete the proof of \eqref{imUe} by showing that $B_{\eta/2}\subset (Id+L_\ve)(B_\eta)$ for $\ve\in (0,\ve_0)$. For this purpose, we let  $\psi_0\in B_{\eta/2}$ and consider the compact operator $T:B_\eta\to \dot C^\beta (\po ; \R)$, $T(\psi):=L_\ve(\psi)-\psi_0$. 
%
We claim that
\begin{equation}
\l{e304}
(\text{Id}+sT)(\psi)\neq 0, \quad \fo s\in [0,1],\ \fo\psi\in\partial B_\eta.
\ee
Indeed, \eqref{e304} is obtained by contradiction. Otherwise, using \eqref{e305}, we obtain, for some $\psi$ such that $\|\psi\|_{C^\beta(\po)}=\eta$:
\begin{equation*}
\eta/2 < \eta - s\|L_\ve(\psi)\|_{C^\beta (\po)} \leq \|\psi+s L_\ve (\psi)\|_{C^\beta (\po)} = \|s\psi_0\|_{C^\beta (\po)}<\eta/2.
\end{equation*}
By \eqref{e304},  the Leray-Schauder degree $\deg (\text{Id}+sT,B_\eta, 0)$ is well defined. By homotopy invariance, we find that 
\begin{equation*}
\deg(\text{Id}+sT,B_\eta, 0) = \deg(\text{Id},B_\eta, 0)=1.
\end{equation*}
As a consequence, the equation $(\text{Id}+T)(\psi)=0$ admits at least a solution $\psi\in B_\eta$. This $\psi$ satisfies  $(\text{Id}+L_\ve)(\psi)=\psi_0$. The proof of \eqref{imUe} is complete.

Let $\ve\in(0, \ve_0)$. Then, by \eqref{imUe}, there exists some $\psi\in \mathfrak{B}$ such that $U_\ve(\psi)=0$. Let $g=g^0 e^{\im\psi}$. Then $u_\ve = u_{\ve,g}\in\mathcal{E}_d$ is a solution of the Ginzburg-Landau equation, and it satisfies the semi-stiff boundary condition
\begin{equation*}
u_\ve \wedge \frac{\partial u_\ve}{\partial\nu} = T_\ve (g) = U_\ve(\psi) = 0\quad \text{on }\partial\Omega.
\end{equation*}
Therefore,  $u_\ve$ is a critical point of $E_\ve$ with prescribed degree $d$.
\end{proof}

\section{Nondegeneracy of domains is stable}
\label{sno}
%
%
In this section we show that, if a domain $\Omega_0$ satisfies the nondegeneracy conditions {\it (ND1)-(ND2)} required in Theorem \ref{t8.1}, then a slightly perturbed domain $\Omega\approx\Omega_0$ still satisfies these nondegeneracy conditions.
\begin{theorem}\label{thmnds}
Assume that $\Omega_0$ satisfies {\it (ND1)-(ND2)}. Fix a conformal representation $f_0:\mathbb{D}\to\Omega_0$. There exists $\delta>0$ such that, for every holomorphic map $f\in C^{1,\beta}(\overline\D)$  satisfying $\|f-f_0\|_{C^{1,\beta}}<\delta$, the domain $\Omega:=f(\mathbb{D})$ satisfies {\it (ND1)-(ND2)}.
\end{theorem}
\begin{proof}
Let $V_\beta$ be as in \eqref{e3.23a}.
We let $\widetilde{g}^0\in C^{1,\beta}(\so;\mathbb{S}^1)/\mathbb{S}^1$ denote the canonical boundary datum associated with $\alpha^0:=f_0^{-1}(a^0)$, so that $\widetilde{g}^0=g^0\circ f_0$.

Since $\Omega_0$ satisfies {\it (ND1)},  we know from Proposition~\ref{p301} that there exist:  a neighborhood $\mathcal{V}_1$ of $f_0$ in $V_\beta$, a neighborhood  $\mathcal{V}_2$ of the origin in $C^{1,\beta}(\so;\R)$, and a smooth map $\alpha:\mathcal{V}_1\times\mathcal{V}_2\to\mathbb{D}_*^k$, such that the following holds. 
For $f\in\mathcal{V}_1$ and $\psi\in\mathcal{V}_2$, let $\Omega=f(\mathbb{D})$ and $g=(\widetilde{g}^0 e^{\im\psi})\circ f^{-1}$. Then $a(f,\psi):=f\left(\alpha(f,\psi)\right)$ is a nondegenerate critical point of $W^{\Omega}(\cdot,g)$.

By the above, we may define, as in \eqref{defT*},  the smooth operator $U_{*,f} = U_{*,a(f,0),\widetilde{g}^0\circ f^{-1}}$,
\begin{equation}\label{U*f}
U_{*,f}(\psi)=N^\Omega \left(a(f,\psi\circ f),(\widetilde{g}^0\circ f^{-1}) e^{\im\psi}\right)\quad\text{for small }\psi\in C^{1,\beta}(\partial\Omega;\R)/\R.
\end{equation}
The spaces between which $U_{*,f}$ is defined vary with $f$. In order to deal with  fixed spaces, we consider the linear isomorphisms
\begin{gather}
\label{linisom1}
\Theta_f : \dot{C}^\beta(\partial\Omega;\R)\to\dot{C}^\beta(\so;\R),\,\psi\mapsto |f'|\psi\circ f,\\
\label{linisom2}
\Xi_f : C^ {1,\beta}(\partial\Omega;\R)/\R\to C^ {1,\beta}(\so;\R)/\R,\,\psi\mapsto\psi\circ f,
\end{gather}
and we let 
\begin{equation}\label{U1}
U(f,\psi) = \Theta_f \circ U_{*,f} \circ \Xi_f^{-1} (\psi) \quad\text{for }(f,\psi)\in\mathcal{V}_1 \times (\mathcal{V}_2/\R),
\end{equation}
so that $U_{*,f}$ is a local diffeomorphism if and only if $U(f,\cdot)$ is a local diffeomorphism.

Moreover, if we express $N^\O$ using \eqref{transformN}, then we obtain
\begin{equation}\label{U2}
U(f,\psi) = N^{\mathbb D}(\alpha(f,\psi),\widetilde{g}_0e^{\im\psi}).
\end{equation}
By combining \eqref{U2} with  the explicit formula \eqref{Ndisc} for $N^{\mathbb D}$, we find that $U:\mathcal{V}_1 \times (\mathcal{V}_2/\R)\to \dot{C}^\beta(\so;\R)$ is smooth.

On the other hand, using the definition \eqref{e2.16} 
 of the canonical boundary datum, we have
\begin{equation*}
u_{*,a,g^a}\wedge\frac{\partial u_{*,a,g^a}}{\partial\nu}  = \frac{\partial \Phi_{a,g^a}}{\partial\tau} = \frac{\partial \widehat{\Phi_a}}{\partial\tau} = 0,
\end{equation*}
so that $U(f_0,0)=0$.

Moreover, since $\Omega_0$ satisfies {\it (ND2)}, $U(f_0,\cdot)$ is a local diffeomorphism near the origin, i.e.,  $D_\psi U(f_0,0)$ is invertible. By the implicit function theorem, possibly after shrinking  $\mathcal{V}_1$, for every $f\in\mathcal{V}_1$ there exists $\psi(f)\in\mathcal{V}_2$ such that
\begin{equation}\label{implicitU}
U\left(f,\psi(f)\right) = 0.
\end{equation}
In addition, the map $f\mapsto\psi(f)$ is smooth and we can assume that $D_\psi U (f,\psi(f))$ is invertible. 

Let $f\in \mathcal{V}_1$ and set $\O:=f(\mathbb{D})$. We claim that $\Omega$ satisfies {\it (ND1)-(ND2)}. Assuming the claim proved for the moment, we complete the proof of Theorem \ref{thmnds} by taking any $\delta>0$ such that 
\bes
\{ f\in X_\beta;\, \|f-f_0\|_{C^{1,\beta}}<\delta\}\subset\mathcal{V}_1.
\ees
We next turn to the proof of the claim. 
Let $g_\Omega:=(\widetilde{g}_0 e^{\im\psi(f)})\circ f^{-1}\in C^{1,\beta}(\partial\Omega;\mathbb{S}^1)$, and $a_\Omega:=a(f,\psi(f))\in\Omega_*^k$. By the definition \eqref{implicitU}  of $\psi(f)$ and the definition  \eqref{U1} of $U$, we obtain  
\begin{equation}
\l{e308}
U_{*,f}(\psi(f)\circ f^{-1}) = 0.
\end{equation}
By combining \eqref{e308} with the definition \eqref{U*f}  of $U_{*,f}$, we find that
\begin{equation}
\l{e309}
N^\Omega (a_\Omega, g_\Omega ) = u_{*,a_\Omega,g_\Omega}\wedge\frac{\partial u_{*,a_\Omega,g_\Omega} }{\partial\nu}= \frac{\partial \Phi_{a_\Omega, g_\Omega}}{\partial\tau} = 0.
\end{equation}
  The normalization condition in \eqref{e2.5} combined with \eqref{e309} implies that  $\Phi_{a_\Omega,g_\Omega} = 0$ on $\partial\Omega$, and thus 
\begin{equation}
\l{e310}
\Phi_{a_\Omega,g_\Omega} = \widehat{\Phi}_{a_\Omega}.
\end{equation}
In turn, \eqref{e310} implies that $g_\Omega = g^{a_\Omega}$ is the canonical boundary data associated with $a_\Omega$. Since, by definition of the map $(f,\psi)\mapsto a(f,\psi)$, the configuration $a_\Omega$ is a nondegenerate critical point of $W(\cdot,g_\Omega)$, we find that the nondegeneracy condition {\it (ND1)} is satisfied by $\O$.

On the other hand, since $D_\psi U (f,\psi(f))$ is invertible, $U_{*,f}$ is a local diffeomorphism near $\psi(f)\circ f^{-1}$, which means that $U_{*,a_\Omega,g_\Omega}$ is a local diffeomorphism near the origin. We find that  $\O$ satisfies {\it (ND2)}.

The proof of Theorem \ref{thmnds} is complete.
\end{proof}

\section{The radial configuration is nondegenerate}
\label{sra}

In this section we let $d=1$, $k=1$, $d_1=1$,  and prove that the unit disc $\mathbb{D}$ satisfies {\it (ND1)-(ND2)}. As a consequence,  domains close to the unit disc satisfy the nondegeneracy conditions when $d=1$, $k=1$, $d_1=1$.

\begin{proposition}\label{propdisc}
Assume $\Omega=\mathbb{D}$, $k=1$, $d=1$. Then $a=0$ is a nondegenerate critical point of $W(\cdot,g^0)$, and $DU_{*,0,g^0}(0)$ is invertible.
\end{proposition}
\begin{proof}
{\it Step 1.} $0$ is a nondegenerate critical point of $W(\cdot,g^0)$.\\
Indeed, by combining \eqref{aba1} and \eqref{hatphidisc}, we easily obtain that  the canonical boundary datum $g^0:\so\to\so$ corresponding to $a=0$ is given by $g^0(z)=z$.  From \eqref{e3.31} we know that
\be\label{e4.4}
W(a,g^0) = \widehat{W}(a)+\frac{1}{2}\int_{\mathbb D}\left|\nabla\psi_{a,g^0}^*\right|^2.
\ee
On the other hand,  \eqref{e3.40} leads to
\be\label{e4.5}
\nabla\psi_{a,g^0}^*(x) = -2\frac{a(\overline{a}x-1)}{|1-\overline{a}x|^2},\quad\fo x\in\D,
\ee
and therefore
\be\label{e4.6}
\frac{1}{2}\int_{\mathbb D}\left|\nabla\psi_{a,g^0}^*\right|^2 = 2|a|^2\int_{\mathbb{D}}\frac{dx}{|1-\overline{a}x|^2}.
\ee
Thanks to the $|a|^2$ factor, if we differentiate \eqref{e4.6} with respect to $a$, and next let $a=0$, we obtain
\begin{equation}\label{e4.6bis}
\left. \nabla_a \left[\frac{1}{2}\int_{\mathbb D}\left|\nabla\psi_{a,g^0}^*\right|^2 \right]\right|_{a=0} =0.
\end{equation}
If we differentiate twice \eqref{e4.6} with respect to $a$, and next let $a=0$, then we are left with only one non zero term  (thanks to the $|a|^2$ factor again). More specifically, we obtain
\be\label{e4.7}\begin{split}
\left. \nabla_a^2\left[\frac{1}{2}\int_{\mathbb D}\left|\nabla\psi_{a,g^0}^*\right|^2 \right]\right|_{a=0} 
& = 4 \left. \int_{\mathbb{D}}\frac{dx}{|1-\overline{a}x|^2}\right|_{a=0}I_2  = 4 \int_{\mathbb{D}} dx \: I_2 = 4\pi I_2.
\end{split}\ee
By combining \eqref{e4.4} with \eqref{e4.6bis} and \eqref{e4.7}, we find that
\begin{equation}\label{e4.8}
\nabla_a W(0,g^0) = \nabla \widehat{W}(0), \quad
\nabla_a^2 W(0,g^0) = \nabla^2\widehat{W}(0)+4\pi I_2.
\end{equation}
We next compute $\nabla\widehat{W}(0)$ and $\nabla^2\widehat{W}(0)$. When $k=1$ and $d=1$, formula \eqref{hatwdisc} reads
\be
\label{e4.1}
\widehat{W}(a)=\pi\log\left( 1-|a|^2\right),\quad\forall\,  a\in\mathbb{D}.
\ee
 Identifying the complex number $a$ with a vector in $\R^2$,  the two first derivatives of $\widehat{W}$ are respectively given by:
\begin{gather}
\label{e4.2}
\nabla\widehat{W}(a)  = \frac{2\pi}{|a|^2-1}a\in\R^2\\
\label{e4.3}
\nabla^2\widehat{W}(a)  = \frac{2\pi}{|a|^2-1}I_2-\frac{4\pi}{\left(|a|^2-1\right)^2}a\otimes a \in M_2(\R).
\end{gather}
In particular, we obtain $\nabla\widehat{W}(0)=0$ and $\nabla^2\widehat{W}(0)=-2\pi I_2$. Plugging this into \eqref{e4.8} yields 
\be
\l{e310a}
\nabla_a W(0,g^0)=0, \quad \nabla_a^2 W(0,g^0)=2\pi I_2,
\ee 
so that $a=0$ is indeed a nondegenerate critical point of $W(\cdot,g^0)$.

\medskip
\noindent
{\it Step 2.} $DU_*(0)$ is invertible.\\
In our case, formula \eqref{Ndisc} becomes
\begin{equation}\label{Ndisc0}
N(a,g^0 e^{\im\psi}) = \frac{\partial\psi^*}{\partial\tau}-2 \frac{a\wedge z}{|z-a|^2}.
\end{equation}
Therefore
\begin{equation}\label{U*disc0}
U_*(\psi) = \frac{\partial\psi^*}{\partial\tau}-2 \frac{a(\psi)\wedge z}{|z-a(\psi)|^2},
\end{equation}
where $\psi\mapsto a(\psi)$ is smooth, $a(0)=0$ and $a(\psi)$ is a nondegenerate critical point of $W(\cdot,g^0 e^{\im\psi})$. 

Using \eqref{U*disc0} together with the fact that $a(0)=0$, we obtain that
\begin{equation}\label{DU*(0)}
DU_*(0)\psi = \frac{\partial\psi^*}{\partial\tau}-2(Da(0)\psi)\wedge z.
\end{equation}
In \eqref{DU*(0)}, $\psi$ is either a function in $C^{1,\beta}(\so;\R)$, or a class in $C^{1,\beta}(\so;\R)/\R$.
Thus the linear operator $DU_*(0):C^{1,\beta}(\so;\R)/\R \to \dot{C}^\beta(\so;\R)$ can be written $DU_*(0)=L-K$, where
\begin{equation*}
L(\psi):=\frac{\partial\psi^*}{\partial\tau}\ \text{and}\  K(\psi):=2(Da(0)\psi)\wedge z,\quad \fo \psi\in C^{1,\beta}(\so;\R)/\R.
\end{equation*}
The operator $L$ is an isomorphism, and $K$ is compact since it has finite range. As a consequence, $DU_*(0)$ is Fredholm of index zero and, in order to complete   Step {\it 2},  it suffices to prove that $DU_*(0)$   is injective.
For this purpose, we compute $Da(0)$ using the implicit equation 
\begin{equation}\label{aimplicit}
F(a(\psi),\psi):=\nabla_a W( a(\psi), g^0 e^{i\psi}) = 0
\end{equation}
satisfied by $a$.
By differentiating \eqref{aimplicit} with respect to $\psi$ we obtain (via \eqref{e310a})
\begin{equation}\label{Da1}
D_\psi F(0,0)\psi = -\nabla_a^2 W(0,g^0) Da(0)\psi = -2\pi  Da(0)\psi.
\end{equation}
Let us compute $D_\psi F(0,0)$. Recalling \eqref{Wdisc}, we find that
\begin{equation}\label{F(a,psi)}
\begin{split}
F(a,\psi) = & \nabla\widehat{W}(a)
+\nabla_a\left[\frac{1}{2}\int_{\mathbb D} \left|\nabla\psi_{a,g^0}^\ast+\nabla\psi^\ast\right|^2\right] \\
= & \nabla\widehat{W}(a) + \nabla_a\left[ \frac{1}{2}\int_{\mathbb D} \left|\nabla\psi_{a,g^0}^*\right|^2 \right] + \nabla_a\left[\int_{\mathbb D}\nabla\psi^*_{a,g^0}\cdot\nabla\psi^*\right].
\end{split}\end{equation}
The two first terms do not depend on $\psi$, and the last term depends linearly on $\psi$. Hence we obtain
\begin{equation}\label{DpsiF1}
D_\psi F(a,0)\psi = 
\nabla_a\left[\int_{\mathbb D}\nabla\psi^*_{a,g^0}\cdot\nabla\psi^*\right].
\end{equation}
Integrating by parts, using the explicit formula \eqref{e3.40} for $\psi^*_{a,g^0}$, and the fact that 
\bes
\int_{\so}\frac{\partial \psi^*}{\p\nu} = \int_{\so}\frac{\partial\psi}{\partial\tau} = 0,
\ees 
we find that
\begin{equation}\label{DpsiF2}
\int_{\mathbb D}\nabla\psi^*_{a,g^0}\cdot\nabla\psi^* = - 2\int_{\so} \log |1 - \overline{a}z|\frac{\partial\psi^*}{\partial\nu}.
\end{equation}
If we first plug \eqref{DpsiF2} into \eqref{DpsiF1} and next let $a=0$, then we obtain
\begin{equation}\label{DpsiF3}
D_\psi F(0,0)\psi = 2\int_{\so} z \frac{\partial\psi^*}{\partial\nu},
\end{equation}
and finally, using \eqref{Da1},
\begin{equation}\label{Da2}
Da(0)\psi = -\frac{1}{\pi}  \int_{\so} z \frac{\partial\psi^*}{\partial\nu}.
\end{equation}
We are now in  position to prove that $DU_*(0)$ is injective. Let $\psi\in \ker DU_*(0)$.  Then, recalling \eqref{DU*(0)}, we have
\begin{equation}\label{kerDU*}
\frac{\partial\psi^*}{\partial\tau} = 2(Da(0)\psi)\wedge z = \alpha\wedge z,
\end{equation}
where
\begin{equation}\label{alpha}
\alpha = -\frac{2}{\pi}  \int_{\so} z \frac{\partial\psi^*}{\partial\nu} \in\C.
\end{equation}
Since $\psi^*$ is harmonic and has zero average  on $\so$, we may write
\begin{equation}\label{psi*fourier}
\psi^*(re^{\im\theta})=\sum_{n\neq 0} a_n r^n e^{\im n\theta}.
\end{equation}
Hence \eqref{kerDU*} yields
\begin{equation}\label{kerDU*2}
\frac{\overline{\alpha}}{2\im}e^{\im\theta}-\frac{\alpha}{2\im}e^{-\im\theta}=\frac{\partial\psi^*}{\partial\tau}(e^{\im\theta}) = \sum_{n\neq 0} \im na_n e^{\im n\theta}.
\end{equation}
Identifying the Fourier coefficients, we obtain
\begin{equation}\label{coef}
a_n=0 \text{ for }|n|>1,\quad a_1 = -\frac{\overline{\alpha}}{2},\quad a_{-1}=-\frac{\alpha}{2},
\end{equation}
so that \eqref{psi*fourier} becomes
\begin{equation}\label{kerDU*3}
\psi^*(re^{\im\theta})=-\frac{\overline{\alpha}}{2} re^{\im\theta} -\frac{\alpha}{2} re^{-\im\theta}.
\end{equation}
By \eqref{kerDU*3}, we have 
\begin{equation}\label{intzdnupsi}
\int_{\so}z\frac{\partial\psi^*}{\partial\nu}  = -\frac{1}{2}\int_0^{2\pi} e^{i\theta}(\overline{\alpha}e^{\im\theta}+\alpha e^{-\im\theta}) d\theta  = -\pi\alpha.
\end{equation}
Plugging \eqref{intzdnupsi} into \eqref{alpha} we obtain $\alpha=2\alpha$, so that $\alpha=0$ and consequently $\psi^*=0$. Therefore, we have  $\psi=0$ modulo $\R$, and thus 
$DU_*(0)$ is invertible.
\end{proof}
\begin{corollary}
If a domain $\Omega$ is sufficiently  close  to the unit disc, in the sense that there exists a conformal representation $f:\mathbb{D}\to\Omega$ such that $\|f-\text{\rm Id}\|_{C^{1,\beta}}<\delta$ for sufficiently small $\delta$, then,  for small $\ve$, $E_\ve$ admits critical points with prescribed degree one. 
\end{corollary}

\section{In degree one, \enquote{most} of the domains are non degenerate}
\l{gen}

In this section, we assume that $k=1$ and $d=1$, and we prove that every domain can be approximated with  domains satisfying the nondegeneracy conditions {\it (ND1)-(ND2)}. More specifically, we establish the following result.
\begin{theorem}\label{ndg}
Assume that $k=1$ and $d=1$. Let $\Omega_0\subset\R^2$ be a simply connected bounded domain with $C^{1,\beta}$ boundary, and fix a conformal representation $f_0:\mathbb{D}\to\Omega_0$. 

Then, for every $\eta>0$, there exists a conformal representation $f:\mathbb{D}\to\Omega:=f(\D)$ such that $\|f_0-f\|_{C^{1,\beta}}<\eta$ and such that the corresponding domain  $\Omega$ satisfies (ND1)-(ND2).
\end{theorem}

The main idea of the proof of Theorem~\ref{ndg} is to use transversality. Among other ingredients, we will rely on the following abstract transversality result \cite[Theorem 3]{quinn70}.
\begin{theorem}\label{transv}
Let $X$, $\Lambda$, $Y$ be smooth separable Banach manifolds. Let $\Phi:X\times \Lambda \to Y$ be a smooth map.

Assume that:\\
1. for every $\lambda\in\Lambda$, $\Phi_\lambda:=\Phi(\cdot,\lambda):X\to Y$ is Fredholm.\footnote{That is, the linearized operator $D_x\Phi(x,\lambda)$ is Fredholm for every $x$ and every $\lambda$.}\\
2. $\Phi$ is transverse to $\lbrace 0\rbrace$, i.e.,  for every $(x,\lambda)$ such that $\Phi(x,\lambda)=0$, the differential $D\Phi(x,\lambda)$ is onto.

Then the set
$\left\lbrace  \lambda ; \: \Phi_\lambda \text{ is transverse to }\lbrace 0\rbrace \right\rbrace$
is dense in $\Lambda$.
\end{theorem}

Note that, if $X$ and $Y$ are finite dimensional, then condition {\it 1.}  is automatically satisfied.

Another ingredient of the proof is the following fact, which relates non degenerate critical points of $\widehat{W}$ to non degenerate critical points of $W(\cdot,g^a)$.

\begin{proposition}\label{magic}
Assume that $k=1$ and $d=1$. Let $a_0\in\Omega$ be a non degenerate critical point of $\widehat{W}^\Omega$. Then $a_0$ is a non degenerate critical point of $W(\cdot,g^{a_0})$.
\end{proposition}
\begin{proof}
Let us first remark that $a_0$ is automatically a critical point of $W^{\Omega_0}(\cdot,g^{a_0})$.\footnote{This is not specific to the case where $k=1$ and $d=1$, but holds for arbitrary $k$ and degrees $d_j$, $j\in\llbracket 1,k\rrbracket$.} Indeed, using \eqref{e2.24}, in which each term is smooth thanks to the formulas in Sections~\ref{st} and \ref{sef}, and the fact that (by definition) we have $\psi_{a,g^a}=0$, we find that
\begin{equation}\label{magic0}
\nabla_a W(a,g)|_{g=g^a} = \nabla\widehat{W}(a).
\end{equation}

It remains to prove that $a_0$ is non degenerate as a critical point of $W^\Omega(\cdot,g^{a_0})$.

Let $f:\mathbb{D}\to\Omega$ be a conformal representation  and set  $\alpha_0:=f^{-1}(a_0)$. Then $\widetilde{f}(0)=a_0$, where
\begin{equation*}
\widetilde{f}(z)=f\left( \frac{ z+\alpha_0}{1+\overline{\alpha_0}z}\right).
\end{equation*}
Therefore, by replacing  $f$ with $\widetilde f$, we may actually assume that $f(0)=a_0$. In view of \eqref{transformga} and of the fact that, in the unit disc, we have $g^0=\textrm{Id}$, we obtain
\be
\l{ma1}
g^{a_0}\circ f=g^0=\textrm{Id}.
\ee
Recall  that, by Lemmas \ref{transformW} and \ref{discexplicit} and by \eqref{hatwdisc} we have
\begin{gather}
\label{magic1a}
\widehat{W}^\Omega(f(\alpha))=\widehat{W}^{\mathbb D}(\alpha)+P(\alpha,f),\\
\l{magic1b}
W^\Omega(f(\alpha),g^{a_0}) = W^{\mathbb D}(\alpha,g^0)+P(\alpha,f),
\end{gather}
where
\be
\l{magic1c}
\widehat W^\D(\alpha)=\pi\log (1-|\alpha|^2),\ P(\alpha,f):=\pi\log |f'(\alpha)|
\ee
and
\be
\l{magic1d}
W^\D(\alpha,g^0)\text{ is given by \eqref{Wdisc} with }\psi=0. 
\ee
By \eqref{magic1a}-\eqref{magic1d}  and the discussion at the beginning of the proof of Proposition \ref{propndstable}, the assumption that $a_0$ is a non degenerate critical point of $\widehat W^\Omega$ is equivalent to the fact that $0$ is a non degenerate critical point of $\widehat W^\D+P(\cdot, f)$. Similarly, the desired conclusion   (that $a_0$ is a nondegenerate critical point of $W^\O(\cdot, g^{a_0})$) is equivalent to the fact that 
$0$ is non degenerate as a critical point of $W^{\mathbb D}(\cdot,g^0)+P(\cdot,f)$.

Since
\begin{equation}\label{magic2}
\nabla\left[ \widehat{W}^\mathbb{D}+P(\cdot,f) \right](\alpha) = \frac{-2\pi\alpha}{1-|\alpha|^2}+\pi \frac{\overline{f''(\alpha)}}{\overline{f'(\alpha)}}\in\mathbb{C}\simeq\R^2,
\end{equation}
and since $0$ is a critical point of $\widehat{W}^{\mathbb{D}}+P(\cdot,f)$, we have $f''(0)=0$. 

In order to calculate the Hessian of $P(\cdot,f)$ at the origin, we find the second order Taylor expansion of $P(\cdot,f)$:
\begin{equation}\label{magic3}
\begin{split}
P(\alpha,f) & =\pi\log \left|f'(0)+f^{(3)}(0)\alpha^2 + o(|\alpha|^2)\right| \\
& = \pi \log |f'(0)| + \frac{\pi}{2}\log \left( \left|1+\frac{f^{(3)}(0)}{f'(0)}\alpha^2+o(|\alpha |^2)\right|^2 \right) \\
& = P(0,f) + \frac{\pi}{2}\log\left( 1+2\Re\left(\frac{f^{(3)}(0)}{f'(0)}\alpha^2\right)  +o(|\alpha|^2) \right) \\
& = P(0,f) + \frac{\pi}{2}\left( 2\Re\left(\frac{f^{(3)}(0)}{f'(0)}\alpha^2\right)  +o(|\alpha|^2)\right)\\
& = P(0,f) + \pi \left(\frac{\overline{f^{(3)}(0)}}{\overline{f'(0)}}\overline{\alpha} \right)\cdot\alpha +o(|\alpha|^2).
\end{split}
\end{equation}
In the last equality, $z\cdot w$ stands for the real scalar product of the complex numbers $z$ and $w$ (identified with vectors in $\R^2$). 
As a consequence, we have
\begin{equation}\label{magic4}
\nabla_\alpha^2 P(0,f) = \pi M_{f^{(3)}(0)/f'(0)},
\end{equation}
where, for a complex number $z\in\mathbb{C}$, $M_z$ denotes the matrix corresponding to the $\mathbb{R}$-linear map \
\bes
T:\mathbb{C}\to\mathbb{C},\ \xi\xmapsto{T}\overline{z\xi},
\ees
i.e.,
\begin{equation*}
M_z=\left(\begin{array}{cc} \Re\: z & -\Im\: z \\ -\Im\: z & -\Re\: z \end{array}\right).
\end{equation*}
Recall that, from \eqref{e4.3} and \eqref{e310a},   it holds
\begin{equation}\label{magic5}
\nabla^2\widehat{W}^{\mathbb D}(0)=-2\pi I_2 \quad\text{and}\quad \nabla_\alpha^2 W(0,g^0) = 2\pi I_2.
\end{equation}
By combining \eqref{magic4} with \eqref{magic5}, we obtain 
\begin{gather}\label{magic6}
\nabla^2\left[ \widehat{W}^\mathbb{D}+P(\cdot,f) \right](0) = \pi M_{f^{(3)}(0)/f'(0)} - 2\pi I_2, \\
\label{magic7}
\nabla^2\left[W^\mathbb{D}(\cdot,g^0)+P(\cdot,f)\right](0) = \pi M_{f^{(3)}(0)/f'(0)} + 2\pi I_2.
\end{gather}
We claim that the two Hessian matrices \eqref{magic6} and \eqref{magic7}  have the same determinant. In fact, for every $z\in\mathbb{C}$, we have
\bes
\begin{aligned}
&\det \left( M_z-2 I_2\right) 
=\left| \begin{matrix} \Re  z -2 & -\Im  z\\
-\Im z & -\Re  z -2 \end{matrix}\right| = (2 - \Re  z)(\Re z +2)- (\Im z)^ 2 \\
& = \left| \begin{matrix}2 + \Re z & -\Im z \\
-\Im z & 2-\Re z \end{matrix}\right|
 = \mathrm{det}\left(2 I_2 + M_z \right).
\end{aligned}
\ees
The Hessian matrix in \eqref{magic6} being non degenerate by assumption, so is the Hessian in \eqref{magic7}. Therefore $0$ is a non degenerate critical point of $W^{\mathbb D}(\cdot,g^0)+P(\cdot,f)$, which means that $a_0$ is a non degenerate critical point of $W^\Omega(\cdot,g^{a_0})$.
\end{proof}

Before proceeding to the proof of Theorem~\ref{ndg}, we introduce some notation.
For $\alpha\in\mathbb{D}$ and $f\in V_\beta$, let 
\begin{equation}\label{ndg1}
\widehat{F}(\alpha,f)=\nabla_\alpha\left[ \widehat{W}^{\mathbb D}(\alpha)+\pi \log |f'(\alpha)| \right],
\end{equation}
so that $\widehat{F}:\mathbb{D}\times V_\beta \to\R^2$ is smooth (thanks to the computations in Lemma~\ref{wsmooth}), and, by Lemma~\ref{transformW}, a point $a=f(\alpha)\in\Omega=f(\mathbb{D})$ is a non degenerate critical point of $\widehat{W}^\Omega$ if and only if $\alpha$ is a non degenerate zero of $\widehat{F}(\cdot,f)$.

Similarly, $g_0\in C^{1,\beta}(\mathbb{S}^1;\mathbb{S}^1)$ being fixed, we define, for $\alpha\in\mathbb{D}$, $f\in V_\beta$ and $\psi\in C^{1,\beta}(\mathbb{S}^1;\R)$,
\begin{equation}\label{ndg2}
F(\alpha,\psi,f)=\nabla_\alpha\left[ W^{\mathbb D}(\alpha,g_0 e^{\im\psi})+\pi \log |f'(\alpha)| \right],
\end{equation}
so that $F:\mathbb{D}\times C^{1,\beta}(\mathbb{S}^1;\R)\times V_\beta \to \R^2$ is smooth. By Lemma \ref{transformW} and the discussion at the beginning of the proof of Proposition \ref{propndstable},   a point $a=f(\alpha)\in\Omega=f(\mathbb{D})$ is a non degenerate critical point of $W^{\Omega}(\cdot,(g_0 e^{\im\psi} )\circ f^{-1})$ if and only if $\alpha$ is a non degenerate zero of $F(\cdot,\psi,f)$.

Using \eqref{Wdisc}, we may split 
\begin{equation}\label{ndg3a}
\widehat{F}(\alpha,f) = F_1(\alpha)+F_2(\alpha,f)
\ee
and
\be
\l{ndg3b} 
F(\alpha,\psi,f) =F_1(\alpha)+F_2(\alpha,f)+F_3(\alpha,\psi),
\end{equation}
where the smooth maps $F_1$, $F_2$ and $F_3$ are respectively given by
\begin{align}
&F_1(\alpha) =\nabla\widehat{W}^{\mathbb D}(\alpha),\label{ndg4} \\
&F_2(\alpha,f) = \nabla_\alpha\left[ \pi\log|f'(\alpha)| \right] = \pi\frac{\overline{f''(\alpha)}}{\overline{f'(\alpha)}}\in\C\simeq\R^2, \label{ndg5}\\
&F_3(\alpha,\psi)  = \nabla_\alpha\left[ \frac{1}{2}\int_{\mathbb D} \left|\nabla(\psi^*_{\alpha,g_0}+\psi^*)\right|^2 \right].\label{ndg6}
\end{align}
\begin{proof}[Proof of Theorem~\ref{ndg}]
The proof is divided into two steps. In each step we apply the abstract transversality result (Theorem~\ref{transv}) in order to prove  that a certain nondegeneracy is generic.

\smallskip
\noindent
{\it Step 1.}  We may assume that $\widehat{W}^{\Omega_0}$ has a non degenerate critical point  $a_0\in\Omega_0$.\\
Indeed, we claim that $\widehat{F}$ is transverse to $\lbrace 0\rbrace$. This will follow if we prove that  $D_f \widehat{F}(\alpha,f)$ is surjective for every $(\alpha,f)$. In turn, surjectivity is established as follows. For every $h\in X_\beta$ we have
\begin{equation}\label{ndg7}
D_f\widehat{F}(\alpha,f)\cdot h = D_f F_2(\alpha,f)\cdot h =
\pi \frac{\overline{f'(\alpha)h''(\alpha)-f''(\alpha)h'(\alpha)}}{\overline{f'(\alpha)}^2}\in\C\simeq\R^2.
\end{equation}
If $f''(\alpha)\neq 0$, then the choice $h(z)=-\lambda z$ (with  $\lambda\in\C$ arbitrary constant) leads to 
\begin{equation*}
\pi\frac{\overline{f''(\alpha)}}{\overline{f'(\alpha)}^2}\overline{\lambda}\in \mathrm{range}\:D_f\widehat{F}(\alpha,f),
\end{equation*}
so $D_f\widehat{F}(\alpha,f)$ is surjective. If $f''(\alpha)=0$, then we take $h(z)=\lambda z^2$ and obtain
\begin{equation*}
\frac{2\pi}{\overline{f'(\alpha)}}\overline{\lambda}\in \mathrm{range}\:D_f\widehat{F}(\alpha,f),
\end{equation*}
and thus the claim is proved.

Therefore we can apply the transversality theorem: we can choose $f$ arbitrarily close to $f_0$, such that $\widehat{F}(\cdot,f)$ is transverse to $\lbrace 0\rbrace$. Thus, by slightly perturbing $f_0$, we may actually assume that $\widehat{F}(\cdot,f_0)$ is transverse to $\lbrace 0\rbrace$.

Since
\begin{equation}\label{ndg8}
\begin{aligned}
\widehat{W}^{\Omega_0}(f_0(\alpha))&=\widehat{W}^{\mathbb D}(\alpha)+\pi \log |f_0'(\alpha)|\\
& 
= \pi \log (1-|\alpha|^2) + \pi \log |f_0'(\alpha)| 
\longrightarrow -\infty\quad\text{as }|\alpha|\to 1,
\end{aligned}
\end{equation}
there exists some $a_0\in\Omega$, such that $\widehat{W}^\O(a_0)=\max_{\Omega_0} \widehat{W}^\O$. Hence $a_0$ is a critical point of $\widehat{W}^{\Omega_0}$, which is equivalent to the fact that $\alpha_0:=f_0^{-1}(a_0)$ is a zero of $\widehat{F}(\cdot,f_0)$. Since the map $\widehat{F}(\cdot,f_0)$ is transverse to $\lbrace 0\rbrace$, its differential is surjective at $\alpha_0$. Therefore,  $\alpha_0$ is a non degenerate zero of $\widehat{F}(\cdot,f_0)$, which means that $a_0$ is a non degenerate critical point of $\widehat{W}^{\Omega_0}$. The proof of Step {\it 1} is complete.

\smallskip
\noindent
{\it Step 2.}  There exists $f$ arbitrarily close to $f_0$, such that $\Omega=f(\mathbb{D})$ satisfies {\it (ND1)-(ND2).}\\
Thanks to Step {\it 1} and Proposition~\ref{magic}, possibly after slightly perturbing  $f_0$, we may assume that there  exists some $a_0=f_0(\alpha_0)\in\Omega_0$, which is a non degenerate critical point of both $\widehat{W}^{\Omega_0}$ and $W^{\Omega_0}(\cdot,g_0)$ (with $g_0=g^{a_0}$).

Since $\widehat{F}(\alpha_0,f_0)=0$, and since $D_\alpha \widehat{F}(\alpha_0,f_0)$ is invertible, we can apply the implicit function theorem to $\widehat{F}$. There exists an open neighborhood $\mathcal{V}_1$ of $f_0$ in  $V_\beta$, and a smooth function $\alpha:\mathcal{V}_1\to\mathbb{D}$, such that, for every $f\in\mathcal{V}_1$ and for every $\alpha$ sufficiently close to $\alpha_0$, we have
\begin{equation}\label{ndg9}
\widehat{F}(\alpha,f)=0 \Longleftrightarrow \alpha=\alpha(f).
\end{equation}
By Proposition \ref{p301} and by the invertibility  of $D_\alpha \widehat{F}(\alpha_0,f_0)$, we 
 may choose the open neighborhood $\mathcal{V}_1$ such that, for every $f\in\mathcal{V}_1$, the point 
 $a=a(f)=f(\alpha(f))\in\Omega=f(\mathbb{D})$ is doubly non degenerate, that is:  non degenerate as a critical point of $\widehat{W}^\Omega$ and non degenerate as a critical point of $W^\Omega(\cdot,g^a)$. In particular, every domain $\Omega=f({\mathbb D})$, with $f\in\mathcal{V}_1$, satisfies {\it (ND1)}.

Again by the second nondegeneracy property of every  $f\in\mathcal{V}_1$, we may consider the map $U_{*,a,g^a}$, defined as in \eqref{U*},  and  corresponding to $a=a(f)$. In order to complete Step {\it 2}, we have to  find some $f$ arbitrarily close to $f_0$, such that the map $U_{*,a,g^a}$ is a local diffeomorphism at the origin.
 To this end we will again rely on the transversality theorem.
More specifically, 
we define, exactly as in formula \eqref{U1} in the proof of Theorem~\ref{thmnds}, the smooth map
\begin{equation}\label{ndg15}
U:\mathcal{V}_1 \times \mathcal{V}_2/\R \longrightarrow \dot{C}^\beta(\mathbb{S}^1;\R).
\end{equation}
Recall that $\mathcal{V}_1$ is an open neighborhood of $f_0$ in $V_\beta$, that  $\mathcal{V}_2$ is an open neighborhood of the origin in $C^{1,\beta}(\mathbb{S}^1;\R)$, and that
\begin{equation}\label{ndg16}
U(f,\psi)=N^{\mathbb D}(\widetilde{\alpha}(\psi,f),g_0 e^{\im \psi})\quad\forall\,  (f,\psi)\in \mathcal{V}_1\times\mathcal{V}_2/\R.
\end{equation}
Here,  $\widetilde{\alpha}$ is the smooth implicit solution of 
\begin{equation}\label{ndg17}
F(\widetilde{\alpha}(\psi,f),\psi,f) = 0
\end{equation}
obtained in Proposition~\ref{propndstable}.
We recall the following fact established in the proof of Theorem~\ref{thmnds}:  the map $U_{*,a(f),g^{a(f)}}$ is a local diffeomorphism at the origin if and only if $U(f, \cdot)$ is a local diffeomorphism at $-\psi_{\alpha(f),g_0}$.

Recalling the formula \eqref{Ndisc} for $N^\mathbb{D}$, we obtain the following explicit formula for $U$:
\begin{equation}\label{ndg18}
U(f,\psi) = \frac{\partial\psi^*}{\partial\tau}+2 \frac{\alpha_0\wedge z}{|z-\alpha_0|^2}-2 \frac{\alpha(\psi,f)\wedge z}{|z-\alpha(\psi,f)|^2}.
\end{equation}
Hence, for every $(f,\psi)\in\mathcal{V}_1\times\mathcal{V}_2/\R$, we have
\begin{equation*}
\begin{aligned}
D_\psi U (f,\psi)\cdot\zeta =& \frac{\partial\zeta^*}{\partial\tau}-2 \frac{\left(D_\psi\widetilde{\alpha}(\psi,f)\cdot\zeta\right)\wedge z}{|z-\widetilde{\alpha}(\psi,f)|^2}\\
&-4\frac{(z-\widetilde{\alpha}(\psi,f))\cdot\left(D_\psi\widetilde{\alpha}(\psi,f)\cdot\zeta\right)}{|z-\widetilde{\alpha}(\psi,f)|^4} \widetilde{\alpha}(\psi,f) \wedge z.
\end{aligned}
\end{equation*}
In particular $D_\psi U (f,\psi)$ is a Fredholm operator of index zero, since it can be written as $L-K$, where 
\bes
L:C^{1,\beta}(\so ; \R)/\R\to\dot C^\beta(\so ; \R),\ \zeta\xmapsto{L}\frac{\p\zeta^*}{\partial\tau}
\ees 
is invertible and $K$ has finite range. Hence  $U(f,\cdot)$ is a smooth Fredholm map for every $f\in\mathcal{V}_1$. 

We want to apply the transversality theorem to $U$. We already know that assumption {\it 1.} of the transverality theorem is satisfied. It 
remains to check that $U$ is transverse to $0$. To this end we compute the differential of $U$ at some point $(f,\psi)$, using \eqref{ndg18}:
\begin{equation}\label{ndg19}
\begin{split}
DU(f,\psi)\cdot (h,\zeta) & = \frac{\partial\zeta^*}{\partial\tau}-2 \frac{\left(D\widetilde{\alpha}(\psi,f)\cdot (h,\zeta)\right)\wedge z}{|z-\widetilde{\alpha}(\psi,f)|^2}\\
& \quad -4\frac{(z-\widetilde{\alpha}(\psi,f))\cdot\left(D\widetilde{\alpha}(\psi,f)\cdot (h,\zeta)\right)}{|z-\widetilde{\alpha}(\psi,f)|^4} \widetilde{\alpha}(\psi,f) \wedge z.
\end{split}
\end{equation}
Let us show that $DU(f,\psi)$ is onto. Let $\Psi\in\dot{C}^\beta(\mathbb{S}^1;\R)$. Then there exists some $\zeta\in C^{1,\beta}(\mathbb{S}^1;\R)/\R$ such that
\begin{equation}\label{ndg20}
\frac{\partial\zeta^*}{\partial\tau}=\Psi.
\end{equation}
We claim that there exists $h=h_\zeta\in X_\beta$ such that
\begin{equation}\label{ndg21}
D\widetilde{\alpha}(\psi,f)\cdot (h_\zeta,\zeta)=0.
\end{equation}
Then, plugging \eqref{ndg21} and \eqref{ndg20} into \eqref{ndg19}, we obtain
\begin{equation*}
DU(f,\psi)\cdot (h_\zeta,\zeta)=\Psi,
\end{equation*}
and thus $DU(f,\psi)$ is onto.

In order to complete Step {\it 2}, it remains 
to prove the existence of $h_\zeta$. From the implicit equation \eqref{ndg17} satisfied by $\widetilde{\alpha}$, we obtain
\begin{equation}\label{ndg22}
\begin{aligned}
&D\widetilde{\alpha}(\psi,f)\cdot (h,\zeta) \\
&=-  D_\alpha F (\widetilde{\alpha}(\psi,f),\psi,f)^{-1}\left[D_f F (\widetilde{\alpha}(\psi,f),\psi,f)\cdot h + D_\psi F (\widetilde{\alpha}(\psi,f),\psi,f)\cdot \zeta  \right].
\end{aligned}
\end{equation}
Since $D_f F (\alpha,\psi,f) = D_f \widetilde F(\alpha,f)$ is surjective (by Step {\it 1}), we may clearly choose $h_\zeta$ such that \eqref{ndg22} holds.

Therefore we can apply the transversality theorem to $U$: the set of $f$ such that $U(f,\cdot)$ is transverse to $\lbrace 0\rbrace$ is dense.

Let $\eta>0$. We can choose $f\in\mathcal{V}_1$, such that $\|f-f_0\|_{C^{1,\beta}}<\eta$, and $U(f,\cdot)$ is transverse to $\lbrace 0\rbrace$. In particular, the differential of $U(f,\cdot)$ at $-\psi_{\alpha(f),g_0}$ is onto, which implies that the differential is invertible (since it is a zero index Fredholm  operator). Hence $U(\cdot,f)$ is a local diffeomorphism at $-\psi_{\alpha(f),g_0}$, which is equivalent to $U_{*,a(f),g^{a(f)}}$ being a local diffeomorphism at the origin, i.e. $\Omega=f(\mathbb{D})$ satisfies {\it (ND2)}.

Step {\it 2} and the proof of Theorem \ref{ndg} are complete.
\end{proof}

\begin{remark}
In Theorem~\ref{ndg} we have established that nondegeneracy of the domain is generic in the case of prescribed degree $d=1$. Some, but not all, of the ingredients of our proof can be generalized to arbitrary $d$.
For example, it is possible to adapt our arguments and obtain the transversality of $\widehat{F}$  to $0$ when $d$ is arbitrary. However, this does not lead to the conclusion that {\it (ND1)} is generically true. The reason is that when $d\neq\pm 1$, we cannot rely on  \eqref{ndg8} anymore, and we actually do not know whether $\widehat W^\O$ does have critical points. 
A similar difficulty occurs in Step {\it 2}. Indeed, the first ingredient in  Step {\it 2} is 
%
Proposition~\ref{magic}, yielding the existence of a  
non degenerate critical point $a_0$ of 
$W(\cdot,g^{a_0})$. Clearly, our proof of Proposition~\ref{magic} is specific to the case $d=1$.

However, it is plausible the the transversality arguments extend to an arbitrary degree $d$, and thus the main difficulty arises in the existence of critical points of $\widehat W^\O$. It would be interesting to investigate, e.g. by topological methods in the spirit of  \cite{able}, whether such points do exist.  
\end{remark}

\section*{Appendix}
\addcontentsline{toc}{section}{Appendix}
The following is a $C^{1,\beta}$ variant of \cite[Lemmas A1, A2]{bbh93}. 
\begin{lemma}\label{estimlemma}
Let $G\subset\R^n$ be a bounded open set of class $C^{1,\beta}$. Assume that
\begin{equation}\label{eqw}
\left\lbrace
\begin{array}{rll}
\Delta w & = f & \text{in }G\\
w & = \varphi & \text{on }\partial G
\end{array}
\right..
\end{equation}
Then
\begin{align}
\sup_{G}|\nabla w| & \leq C\left( \|f\|_{L^\infty}^{1/2}\left(\|w\|_{L^\infty}^{1/2} +\|\va\|_{L^\infty(\p G)}^{1/2}\right)+ \|\varphi\|_{C^{1,\beta}(\partial G)}\right),\\
|\nabla w|_{0,\beta, G} & \leq C\left( \|f\|_{L^\infty}^{1/2+\beta/2}\left(\|w\|_{L^\infty}^{1/2-\beta/2}+\|\va\|_{L^\infty(\p G)}^{1/2-\beta/2}\right) + \|\varphi\|_{C^{1,\beta}(\partial G)}\right),
\end{align}
for a constant $C$ depending only on $G$. In addition, when $G=\Omega_\sigma$, where $\sigma_1\le \sigma\le \sigma_2$ and $\sigma_1$, $\sigma_2$ are fixed small numbers, we may take $C$ independent of $\sigma$. 
\end{lemma}

\begin{proof}
We write $w=u+v$, where
\begin{equation}\label{equ}
\left\lbrace\begin{array}{rll}
\Delta u &= 0 & \text{in }G\\
u&=\varphi & \text{on }\partial G
\end{array},\right.
\end{equation}
and
\begin{equation}\label{eqv}
\left\lbrace\begin{array}{rll}
\Delta v &= f & \text{in }G\\
v&= 0 & \text{on }\partial G
\end{array}\right..
\end{equation}
By  standard elliptic estimates  \cite[Theorem~8.33]{gilbargtrudinger} we have
\begin{equation}\label{estimu}
\| u \|_{C^{1,\beta}}\leq c \|\varphi \|_{C^{1,\beta}}.
\end{equation}
Therefore we only need to prove that $v$ satisfies the estimates
\begin{align}
\label{estimv1}
\sup_{G}|\nabla v| & \leq C \|f\|_{L^\infty}^{1/2}\|v\|_{L^\infty}^{1/2},\\
\label{estimv2}
|\nabla v|_{0,\beta,G} & \leq C \|f\|_{L^\infty}^{1/2+\beta/2}\|v\|_{L^\infty}^{1/2-\beta/2}.
\end{align}
Estimate \eqref{estimv1} is proved in \cite[Lemma~A.2]{bbh93} by combining an interior estimate with a boundary estimate. Estimate \eqref{estimv2} can be obtained following exactly the same lines. In order to see this, we detail for example the proof of the interior estimate corresponding to \eqref{estimv2}. Proceeding as in \cite[Lemma~A.1]{bbh93}, we first show that
\begin{equation}\label{estimv3}
|\nabla v|_{0,\beta,G_d}\leq 
C \left( \|f\|_{L^\infty}^{1/2+\beta/2}\|v\|_{L^\infty}^{1/2-\beta/2}
+\frac{1}{d^{1+\beta}}\|v\|_{L^\infty}\right),
\end{equation}
where, for $d>0$, we let $G_d:=\lbrace x\in G; \, \mathrm{dist}(x,\partial G) > d \rbrace$. In order to prove \eqref{estimv3}, we let $x_0\in G_d$ and $\lambda\in (0,d]$,  and define
\begin{equation}\label{defulambda}
v_{\lambda}(y):=v(x_0+\lambda y),\quad y\in B_1(0).
\end{equation}
Then the function $v_\lambda$ satisfies the equation
\begin{equation}\label{equlambda}
\Delta v_\lambda = f_\lambda  \text{ in }B_1(0),\quad\text{with }f_\lambda(y):=\lambda^2 f(x_0+\lambda y).
\end{equation}
Standard elliptic estimates \cite[Theorem~8.33]{gilbargtrudinger} yield
\begin{equation}\label{estimv4}
\begin{split}
\lambda^{1+\beta}|\nabla v|_{0,\beta,B_{\lambda/2}(x_0)} &= |\nabla v_\lambda|_{0,\beta,B_{1/2}(0)}  
\leq C \left(\|v_\lambda\|_{L^\infty}+\|f_\lambda\|_{L^\infty}\right) \\
& \leq C\left( \|v\|_{L^\infty}+\lambda^2 \|f\|_{L^\infty} \right).
\end{split}
\end{equation}
We next discuss the two following cases.

\noindent
\emph{Case 1.} $\d 
\frac{\|v\|_{L^\infty}}{\|f\|_{L^\infty}}\leq d^2.
$\\
In this case,  we apply \eqref{estimv4} with $\lambda =(\|v\|_{L^\infty}/\|f\|_{L^\infty})^{1/2}$. We find that
\begin{equation}\label{estimcase1}
|\nabla v|_{0,\beta,B_{\lambda/2}(x_0)}\leq 2C \|v\|_{L^\infty}^{1/2-\beta/2}\|f\|_{L^\infty}^{1/2+\beta/2},
\end{equation}
so that \eqref{estimv3} is satisfied.

\noindent
\emph{Case 2.}
$\d
\frac{\|v\|_{L^\infty}}{\|f\|_{L^\infty}}> d^2.
$\\
In this case, we apply \eqref{estimv4} with $\lambda=d$. We obtain
\begin{equation}\label{estimcase2}
\begin{split}
|\nabla v|_{0,\beta,B_{\lambda/2}(x_0)} & \leq C \left( d^{-1-\beta}\|v\|_{L^\infty} + d^{1-\beta}\|f\|_{L^\infty} \right)\\
& \leq C \left(d^{-1-\beta}\|v\|_{L^\infty} + \|v\|_{L^\infty}^{1/2-\beta/2}\|f\|_{L^\infty}^{1/2+\beta/2}\right),
\end{split}
\end{equation}
so that in both cases \eqref{estimv3} is satisfied.

Once \eqref{estimv3} is established, we easily obtain the interior estimate corresponding to \eqref{estimv2}. Indeed, standard elliptic estimates \cite[Theorem~3.7]{gilbargtrudinger} imply
$\|v\|_{L^\infty}\leq C \|f\|_{L^\infty}$, so that from \eqref{estimv3} we obtain
\begin{equation}
\l{ee1}
|\nabla v|_{0,\beta,K}\leq C \|f\|_{L^\infty}^{1/2+\beta/2}\|v\|_{L^\infty}^{1/2-\beta/2},
\end{equation}
for every compact set $K\subset G$.

The proof of the boundary version of \eqref{ee1} is also a straightforward adaptation of the corresponding estimate established in   \cite[proof of Lemma A.2]{bbh93}, and we omit it here.
\end{proof}

\bibliographystyle{plain}
\bibliography{prescribed_degrees}

\end{document}